\documentclass[reqno]{amsart}

\usepackage[english]{babel}

\usepackage[letterpaper,top=2cm,bottom=2cm,left=2cm,right=2cm,marginparwidth=1.75cm]{geometry}

\usepackage{amsmath}
\usepackage{amsfonts}
\usepackage{amsthm}
\usepackage{amssymb}
\usepackage{dsfont}
\usepackage{soul} 
\usepackage[usenames, dvipsnames]{xcolor}

\theoremstyle{plain}
\newtheorem{proposition}{Proposition}%
\newtheorem{lemma}{Lemma}
\newtheorem{theorem}{Theorem}%
\theoremstyle{remark}
\newtheorem{remark}{Remark}%
\newtheorem{example}{Example}%
\newtheorem{definition}{Definition}%

\usepackage{tikz}
\usepackage{mathrsfs} 
\usepackage{graphicx}
\usepackage{bm} 
\usepackage[colorlinks=true,linkcolor=BrickRed,citecolor=RoyalBlue]{hyperref} 
\usepackage{cleveref}

\definecolor{colorOne}{RGB}{0, 39, 77}
\definecolor{colorTwo}{RGB}{97, 0, 57}
\definecolor{colorThree}{RGB}{176, 1, 54}

\newcommand{\timeVariable}{t} 
\newcommand{\spaceVariable}{x} 
\newcommand{\vectorial}[1]{\bm{#1}} 
\newcommand{\conservedVariable}{u} 
\newcommand{\flux}{\varphi} 
\newcommand{\idEst}{\emph{i.e.}}
\newcommand{\confer}{\emph{cf.}}
\newcommand{\spaceStep}{\Delta \spaceVariable} 
\newcommand{\timeStep}{\Delta \timeVariable} 
\newcommand{\latticeVelocity}{\lambda} 
\newcommand{\discreteVelocityLetter}{c} 
\newcommand{\indexVelocity}{i} 
\newcommand{\relatives}{\mathbb{Z}} 
\newcommand{\naturals}{\mathbb{N}} 
\newcommand{\discrete}[1]{\mathsf{#1}} 
\newcommand{\distributionFunctionLetter}{f}
\newcommand{\distributionFunction}{\discrete{\distributionFunctionLetter}} 
\newcommand{\indexSpace}{j} 
\newcommand{\indexTime}{n} 

\newcommand{\definitionEquality}{:=} 
\newcommand{\collided}{\star} 

\newcommand{\relaxationParameterSymmetric}{\omega_{\discrete{s}}} 
\newcommand{\relaxationParameterAntiSymmetric}{\omega_{\discrete{a}}} 
\newcommand{\atEquilibrium}{\textnormal{eq}} 
\newcommand{\conservedVariableDiscrete}{\discrete{\conservedVariable}} 
\newcommand{\equilibriumCoefficientLinear }{\mathscr{L}} 
\newcommand{\differential}{\textnormal{d}} 
\newcommand{\maximumInitialDatum}{\mu_{\infty}} 
\newcommand{\minimumDistribution}{\underline{m}} 
\newcommand{\maximumDistribution}{\overline{m}} 
\newcommand{\lbmScheme}[2]{$\textnormal{D}_{#1}\textnormal{Q}_{#2}$} 
\newcommand{\reals}{\mathbb{R}} 
\newcommand{\initial}{\circ} 

\newcommand{\transpose}[1]{#1^{\textsf{T}}} 
\newcommand{\indicatorFunction}[1]{\mathds{1}_{#1}} 

\newcommand{\expo}{{\rm e}}
\newcommand{\ud}{\frac{1}{2}}
\newcommand{\dt}{\timeStep}

\newcommand{\es}{\varepsilon_{\discrete{s}}}
\newcommand{\ea}{\varepsilon_{\discrete{a}}}
\newcommand{\oa}{\relaxationParameterAntiSymmetric}
\newcommand{\os}{\relaxationParameterSymmetric}
\newcommand{\minvec}{{\bf \minimumDistribution}}
\newcommand{\maxvec}{{\bf \maximumDistribution}}

\newcommand{\exs}{\expo^{-\frac{\dt}{\es}}}
\newcommand{\exa}{\expo^{-\frac{\dt}{\ea}}}

\newcommand{\strong}[1]{\emph{#1}}

\newcommand{\positiveVelocitySymbol}{+}
\newcommand{\negativeVelocitySymbol}{-}
\newcommand{\zeroVelocitySymbol}{\circ}

\title[Convergence of a two-relaxation-times kinetic approximation]{Convergence of a two-relaxation-times kinetic approximation towards the solution of a scalar conservation law}

\author{Denise Aregba-Driollet}
\address{Université de Bordeaux, CNRS, Bordeaux INP, IMB, UMR 5251, 33400 Talence, France.}
\email{aregba@math.u-bordeaux.fr}

\author{Thomas Bellotti}
\address{Université Paris-Saclay, CNRS, CentraleSupélec, Laboratoire EM2C \& Fédération de Mathématiques de CentraleSupélec, 91190, Gif-sur-Yvette, France}
\email{thomas.bellotti@centralesupelec.fr}


\begin{document}
\maketitle

\begin{abstract}
  We introduce a two-relaxation-times (TRT) kinetic approximation for scalar non-linear conservation laws in one space dimension, addressing the convergence of relaxation approximations to entropy solutions. The proposed TRT system, derived from a lattice Boltzmann scheme, generalizes the classical BGK (single-relaxation-time) framework. Requesting quasi-monotonicity of the relaxation operator, we establish global existence of solutions for the TRT system and prove their convergence, using a lattice Boltzmann scheme, to the entropy solution of the original conservation law as the relaxation times vanish.
  Moreover, we propose a qualitative analysis, clarifying the role of relaxation parameters and equilibrium coefficients in shaping solutions, via Chapman-Enskog expansion and looking at travelling wave solutions.
\end{abstract}






\section{Introduction}

In the present contribution, we are concerned with a novel \strong{relaxation approximation} of the following scalar non-linear conservation law, endowed with an initial datum: 
\begin{equation}\label{eq:conservationLaw}
    \begin{cases}
        \partial_{\timeVariable}\conservedVariable(\timeVariable, {\spaceVariable}) + \partial_{\spaceVariable}\flux(\conservedVariable(\timeVariable, {\spaceVariable})) = 0, \qquad \timeVariable>0, \quad &{\spaceVariable} \in \reals, \\
        \conservedVariable(0, {\spaceVariable}) = \conservedVariable^{\initial}({\spaceVariable}), \qquad &{\spaceVariable} \in \reals.
    \end{cases}
  \end{equation}
  The flux $\flux \in C^1(\reals)$ is non-constant and such that $\flux(0) = 0$.
  We assume that $\conservedVariable^{\initial} \in L^\infty(\reals)$,
  and we denote $\maximumInitialDatum\definitionEquality\|\conservedVariable^{\initial}\|_\infty$ the magnitude of the initial datum.
  We also assume non-trivial data by imposing $\maximumInitialDatum>0$.
  The theory of existence and uniqueness of the weak solution to \eqref{eq:conservationLaw} is classical, see \cite{godlewski1991hyperbolic}.
We consider the \strong{two-relaxation-times} (TRT) kinetic approximation
  \begin{equation}\label{system1}
  \left\{
  \begin{aligned}
  &\partial_t \distributionFunctionLetter_{\zeroVelocitySymbol}^{\vectorial{\varepsilon}}=\frac{1}{\es}(\distributionFunctionLetter_{\zeroVelocitySymbol}^{\atEquilibrium}(u^{\vectorial{\varepsilon}})-\distributionFunctionLetter_{\zeroVelocitySymbol}^{\vectorial{\varepsilon}}),\\
  &\partial_t \distributionFunctionLetter_{\positiveVelocitySymbol}^{\vectorial{\varepsilon}}+\latticeVelocity \partial_x \distributionFunctionLetter_{\positiveVelocitySymbol}^{\vectorial{\varepsilon}}=\frac{1}{2}\left(\frac{1}{\es}+\frac{1}{\ea}\right)(\distributionFunctionLetter_{\positiveVelocitySymbol}^{\atEquilibrium}(u^{\vectorial{\varepsilon}})-\distributionFunctionLetter_{\positiveVelocitySymbol}^{\vectorial{\varepsilon}})+\frac{1}{2}\left(\frac{1}{\es}-\frac{1}{\ea}\right)(\distributionFunctionLetter_{\negativeVelocitySymbol}^{\atEquilibrium}(u^{\vectorial{\varepsilon}})-\distributionFunctionLetter_{\negativeVelocitySymbol}^{\vectorial{\varepsilon}}),\\
  &\partial_t \distributionFunctionLetter_{\negativeVelocitySymbol}^{\vectorial{\varepsilon}}-\latticeVelocity \partial_x \distributionFunctionLetter_{\negativeVelocitySymbol}^{\vectorial{\varepsilon}}=\frac{1}{2}\left(\frac{1}{\es}+\frac{1}{\ea}\right)(\distributionFunctionLetter_{\negativeVelocitySymbol}^{\atEquilibrium}(u^{\vectorial{\varepsilon}})-\distributionFunctionLetter_{\negativeVelocitySymbol}^{\vectorial{\varepsilon}})+\frac{1}{2}\left(\frac{1}{\es}-\frac{1}{\ea}\right)(\distributionFunctionLetter_{\positiveVelocitySymbol}^{\atEquilibrium}(u^{\vectorial{\varepsilon}})-\distributionFunctionLetter_{\positiveVelocitySymbol}^{\vectorial{\varepsilon}}),\\
  \end{aligned}
  \right.
\end{equation}
where $\latticeVelocity>0$ is a kinetic velocity, $\ea>0$ and $\es>0$ are relaxation times\footnote{In the following, the fact that the relaxation times are positive is understood.}, $u^{\vectorial{\varepsilon}}=\distributionFunctionLetter_{\zeroVelocitySymbol}^{\vectorial{\varepsilon}}+\distributionFunctionLetter_{\positiveVelocitySymbol}^{\vectorial{\varepsilon}}+\distributionFunctionLetter_{\negativeVelocitySymbol}^{\vectorial{\varepsilon}}$, and we employ three equilibrium functions (also known as ``Maxwellian'' functions), associated to the characteristic velocities $c_{\zeroVelocitySymbol}=0$ and $c_{\pm}=\pm\latticeVelocity$,
\begin{equation}\label{eq:equilibriaForm}
\distributionFunctionLetter_{\zeroVelocitySymbol}^{\atEquilibrium}(u) = (1-2\equilibriumCoefficientLinear ) u
\qquad \text{and} \qquad
\distributionFunctionLetter_{\pm}^{\atEquilibrium}(u) = \equilibriumCoefficientLinear  u \pm \frac{\flux(\conservedVariable)}{2\latticeVelocity},
\end{equation}
where $\equilibriumCoefficientLinear $ is a fixed real coefficient.
Initial distribution functions are taken at equilibrium, that is:
\begin{equation}\label{f0}
  f_i^{\vectorial{\varepsilon}}(\timeVariable=0,x)=\distributionFunctionLetter_i^{\atEquilibrium} (u^{\initial}(x)) \qquad\text{for}\qquad  i=\zeroVelocitySymbol, \positiveVelocitySymbol, \negativeVelocitySymbol.
 \end{equation}
 
Alternatively, \eqref{system1}--\eqref{eq:equilibriaForm} can be recast on the ``\strong{moments}'' $u^{\vectorial{\varepsilon}} = \sum_{\indexVelocity}\distributionFunctionLetter_{\indexVelocity} =  \distributionFunctionLetter_{\zeroVelocitySymbol}^{\vectorial{\varepsilon}}+\distributionFunctionLetter_{\positiveVelocitySymbol}^{\vectorial{\varepsilon}}+\distributionFunctionLetter_{\negativeVelocitySymbol}^{\vectorial{\varepsilon}}$ (zero-order in the velocities), $v^{\vectorial{\varepsilon}} = \sum_{\indexVelocity}c_{\indexVelocity}\distributionFunctionLetter_{\indexVelocity} = \latticeVelocity (\distributionFunctionLetter_{\positiveVelocitySymbol}^{\vectorial{\varepsilon}}-\distributionFunctionLetter_{\negativeVelocitySymbol}^{\vectorial{\varepsilon}})$ (first-order in the velocities, anti-symmetric), and $w^{\vectorial{\varepsilon}} =\latticeVelocity^{-2}\sum_{\indexVelocity}c_{\indexVelocity}^2\distributionFunctionLetter_{\indexVelocity} = \distributionFunctionLetter_{\positiveVelocitySymbol}^{\vectorial{\varepsilon}} + \distributionFunctionLetter_{\negativeVelocitySymbol}^{\vectorial{\varepsilon}}$ (second-order in the velocities, symmetric), yielding
 \begin{equation}\label{eq:systemMoments}
      \left\{
  \begin{aligned}
  \partial_t u^{\vectorial{\varepsilon}}&+\partial_x v^{\vectorial{\varepsilon}}=0,\\
  \partial_t v^{\vectorial{\varepsilon}} &+ \lambda^2 \partial_x w^{\vectorial{\varepsilon}}=\frac{1}{\ea}\left(
    \flux(u^{\vectorial{\varepsilon}})-v^{\vectorial{\varepsilon}}\right),\\
  \partial_t w^{\vectorial{\varepsilon}} &+\partial_x v^{\vectorial{\varepsilon}}=\frac{1}{\es}\left(2\equilibriumCoefficientLinear u^{\vectorial{\varepsilon}}-w^{\vectorial{\varepsilon}}\right).
  \end{aligned}
  \right.
\end{equation} 
This form shows at a glance that, letting $\ea\to 0$, one formally recovers \eqref{eq:conservationLaw}.
Furthermore, it clarifies the subscripts employed for the relaxation times, for $\ea$ pertains to $v^{\vectorial{\varepsilon}}$, which is \strong{anti-symmetric} while switching $\distributionFunctionLetter_{\positiveVelocitySymbol}^{\vectorial{\varepsilon}}$ and $\distributionFunctionLetter_{\negativeVelocitySymbol}^{\vectorial{\varepsilon}}$, whereas $\es$ concerns $w^{\vectorial{\varepsilon}}$, a \strong{symmetric} moment.

  \begin{remark}[Multi-dimensional setting and more distribution functions]
    We only discuss the case of a 1D conservation law \eqref{eq:conservationLaw} for the sake of presentation.
    The multidimensional setting can be analyzed with the same tools as presented in \cite{aregba2026monotonicity}.
    Moreover, the case of more than three discrete velocities, all two-by-two opposite if non-zero, can be handled analogously.
  \end{remark}

Let us now explain the \strong{origin of \eqref{system1}} and recapitulate existing works on the \strong{convergence of relaxation approximations} to \eqref{eq:conservationLaw}.
To the best of our knowledge, \eqref{system1} has been first written down in a recent paper  \cite{aregba2026monotonicity}  by the authors, while tackling the convergence of a lattice Boltzmann scheme with two relaxation parameters.
This form of numerical scheme has been around for quite a long time, see for instance the works of Ginzburg and collaborators \cite{ginzburg2005equilibrium, ginzburg2007lattice, ginzburg2010optimal}.
Concerning the theoretical understanding of relaxation approximations, efforts concentrated on the BGK (for Bhatnagar-Gross-Krook) case\footnote{Also known as SRT (single-relaxation-time) within the lattice Boltzmann community.}. 
This setting corresponds to letting $\es=\ea=\varepsilon>0$, and the fact that the equilibrium functions are non decreasing implies that the system is \strong{quasi-monotone}, so that one can prove \strong{global existence} of the solution of \eqref{system1}--\eqref{eq:equilibriaForm}--\eqref{f0}, and the \strong{convergence} of $u^{\varepsilon}$ to the entropy solution of \eqref{eq:conservationLaw} when $\varepsilon \rightarrow 0$, see \cite{MR1643175}.

\begin{remark}[\lbmScheme{1}{2} model]
  Letting $\equilibriumCoefficientLinear   = \tfrac{1}{2}$, we obtain a \strong{two-velocity model} where $\distributionFunctionLetter_{\zeroVelocitySymbol}^{\vectorial{\varepsilon}} \equiv 0$ (alternatively, $w^{\vectorial{\varepsilon}}\equiv u^{\vectorial{\varepsilon}}$), thanks to \eqref{f0}.
  From \eqref{eq:systemMoments}, we clearly see that $\es$ does not play any role in this case, thus the theory by \cite{MR1643175} fully applies within this setting.
  For this reason, we are mostly interested by the case $\equilibriumCoefficientLinear  \neq \tfrac{1}{2}$.
\end{remark}

In the present contribution, we aim at following this approach: find sufficient conditions of \strong{quasi-monotonicity} (see \Cref{sec:quasiMonotonicity}) for the operator associated with right-hand side of \eqref{system1}, and prove that---under these conditions and for initial data of bounded total variation---\strong{global existence} of $\vectorial{f}^{\vectorial{\varepsilon}} = (f_{\zeroVelocitySymbol}^{\vectorial{\varepsilon}}, f_{\positiveVelocitySymbol}^{\vectorial{\varepsilon}}, f_{\negativeVelocitySymbol}^{\vectorial{\varepsilon}})^T$ and \strong{convergence} of $u^{{\vectorial{\varepsilon}}}$ to the entropy solution of \eqref{eq:conservationLaw} when $(\ea,\es) \rightarrow (0,0)$ hold.
For initial data only in $L^{\infty}$, the same results can be proved upon slightly strengthening the conditions, which is only technical but needed to ensure global existence of $\vectorial{f}^{\vectorial{\varepsilon}}$.

Furthermore, in \Cref{sec:qualitativeFeatures}, we investigate qualitative features of the relaxation system \eqref{system1}--\eqref{eq:equilibriaForm} thanks to a formal \strong{Chapman-Enskog expansion} and by considering particular solutions under the form of \strong{travelling waves}.
By doing so, we aim at clarifying the impact of each relaxation parameter and of $\equilibriumCoefficientLinear$ on the ``shape'' of the solution.

Our main results in the bounded total variation setting can be resumed as follows.
\begin{theorem}[Main results under BV assumption]\label{thm:hugeTheorem}
  Assume that $u^{\initial}\in L^{\infty}(\reals)\cap L^{1}(\reals)\cap \textnormal{BV}(\reals)$.
Let $\equilibriumCoefficientLinear  \in (0, \tfrac{1}{2}]$ and
\begin{align}
    &\text{if}\quad \equilibriumCoefficientLinear   = \tfrac{1}{2}, 
    \quad \text{then}
    \quad 
    \max_{u\in[-\maximumInitialDatum, \maximumInitialDatum]}
    \frac{|\flux^\prime(u)|}{\lambda}\leq 1; \qquad \text{or} \label{eq:monL212}\\
    &\text{if}\quad \equilibriumCoefficientLinear   \neq \tfrac{1}{2}, 
    \quad \text{then}
    \quad 
    \max_{u\in[-\maximumInitialDatum, \maximumInitialDatum]}
    \frac{|\flux^\prime(u)|}{\lambda}\leq \min\left(  \frac{2  \ea \equilibriumCoefficientLinear  }{\es},1-\frac{\ea}{\es}(1-2\equilibriumCoefficientLinear  )\right).\label{eq:monL2Else}
\end{align}
Then, there exists a unique $\vectorial{f}^{\vectorial{\varepsilon}}\in C^0([0,+\infty); L_{\textnormal{loc}}^1(\reals)^3) \cap L^{\infty}(\reals_+\times \reals)^3$ weak solution of \eqref{system1}--\eqref{eq:equilibriaForm}--\eqref{f0}.

Moreover, if \eqref{eq:monL212}--\eqref{eq:monL2Else} hold while $(\ea, \es)\to(0, 0)$, then $u^{\vectorial{\varepsilon}} \definitionEquality \distributionFunctionLetter_{\zeroVelocitySymbol}^{\vectorial{\varepsilon}}+\distributionFunctionLetter_{\positiveVelocitySymbol}^{\vectorial{\varepsilon}}+\distributionFunctionLetter_{\negativeVelocitySymbol}^{\vectorial{\varepsilon}}$ converges, up to subsequence extraction and for the topology of $C^0([0,+\infty); L_{\textnormal{loc}}^1(\reals)^3)$, to the unique weak entropy solution of \eqref{eq:conservationLaw}.
\end{theorem}
\begin{remark}[On the ``choice'' of $\latticeVelocity$]
    Note that inequalities \eqref{eq:monL212}--\eqref{eq:monL2Else} hold true provided that kinetic velocity $\latticeVelocity>0$ is taken \strong{large enough}.
    If one sees the relaxation system \eqref{system1}--\eqref{eq:equilibriaForm}--\eqref{f0} as an ``approximation'' to \eqref{eq:conservationLaw}---which is aimed at converging with $(\ea, \es)$ in the spirit of the second part of \Cref{thm:hugeTheorem}---then $\latticeVelocity$ can be chosen arbitrarily but to satisfy a \strong{sub-characteristic condition} like \eqref{eq:monL212}.
    This fact has been highlighted for quite a long time in the context of relaxation systems, dating back (at least) to \cite{liu1987hyperbolic}.
\end{remark}
\begin{remark}[On the ratio between $\ea$ and $\es$]\label{rem:onTherationOfpar}
    Note that as we want that \eqref{eq:monL2Else} be fulfilled when $(\ea, \es)\to(0, 0)$, this fact imposes a certain way they must \strong{go to zero together}.
    This is due to the fact that the right-hand side of \eqref{eq:monL2Else} depends on the relaxation parameters through their \strong{ratio}. 
    Of course, our condition is \strong{only sufficient} for convergence, and we cannot confirm or exclude convergence in other cases.
    If, while taking the limit, we face $\ea/\es\to 0$,  \eqref{eq:monL2Else} cannot be fulfilled since $\max_{u\in[-\maximumInitialDatum, \maximumInitialDatum]}
    {|\flux^\prime(u)|}/{\lambda}>0$.
    On the other hand, if $\ea/\es\to +\infty$,  \eqref{eq:monL2Else} cannot be fulfilled either, for the same reason: the right-hand side of the inequality tends to $-\infty$.
    This shows that \Cref{thm:hugeTheorem} concerns the case where 
    \begin{equation*}
    \lim_{(\ea, \es)\to (0, 0)}
        \frac{\ea}{\es}= c\in (0, +\infty).
    \end{equation*}
    More precisely, if the Courant number $\max_{u\in[-\maximumInitialDatum, \maximumInitialDatum]}
    {|\flux^\prime(u)|}/{\lambda}>0$ and $\equilibriumCoefficientLinear  $ are considered fixed, and such that they fulfill 
    \begin{equation}\label{eq:minimalL}
      \equilibriumCoefficientLinear   
      \geq 
      \frac{\max_{u\in[-\maximumInitialDatum, \maximumInitialDatum]}
    {|\flux^\prime(u)|}}{2\latticeVelocity},
    \end{equation}
    the \strong{sector of the first quadrant} in the $(\es, \ea)$-plane for which \eqref{eq:monL2Else} holds is the one between the straight lines of slopes
    \begin{equation*}
      \frac{\max_{u\in[-\maximumInitialDatum, \maximumInitialDatum]}
    {|\flux^\prime(u)|}}{2\latticeVelocity\equilibriumCoefficientLinear  }
    \quad 
    \text{(below the BGK line)}
    \qquad \text{and}\qquad 
    \frac{1-\max_{u\in[-\maximumInitialDatum, \maximumInitialDatum]}
    \frac{|\flux^\prime(u)|}{\lambda}}{1-2\equilibriumCoefficientLinear  }
    \quad \text{(above the BGK line)}.
    \end{equation*}
\end{remark}

Before recapitulating the results obtained  for initial data that are not necessarily of bounded total variation, let us comment on the way the proof of \Cref{thm:hugeTheorem} is conducted.
This proof relies on two main steps as below.
\begin{enumerate}
    \item We introduce a numerical scheme for \eqref{system1} with suitable properties that are uniform in $(\ea, \es)$, including a kinetic entropy inequality. 
    We prove \strong{its convergence} as in \cite{aregba2026monotonicity} in the limit of small discretization parameters, and these properties ``pass to the limit'', and thus hold for $\vectorial{f}^{\vectorial{\varepsilon}}$. 
    This achieves the first part of the claim on the global existence of $\vectorial{f}^{\vectorial{\varepsilon}}$, which bypasses the arguments of \cite{MR1409928, MR1643175}. 
    This is discussed in \Cref{sec:numericalScheme}.
    \item We use the properties that passed to the limit to extract a converging subsequence from $(\vectorial{f}^{\vectorial{\varepsilon}})_{\vectorial{\varepsilon}}$ in the limit $(\ea, \es)\to (0, 0)$, see \Cref{sec:convergenceEps}.
\end{enumerate}

The main results that we obtain for initial data being only essentially bounded are as follows.
\begin{theorem}[Main results under $L^{\infty}$ assumption]\label{thm:hugeTheoremLinf}
  Assume that $u^{\initial}\in L^{\infty}(\reals)$.
  Also assume that 
\begin{align}
    \text{if}\quad
    \es<\ea, 
    \quad 
    &\text{then}
    \quad 
    \equilibriumCoefficientLinear\in 
    \Bigl [
    1-\frac{\es}{\ea}
    , \frac{1}{2}\Bigr ]; \quad \text{or} \label{eq:technical1}\\
    \text{if}\quad
    \es\geq \ea, 
    \quad 
    &\text{then}
    \quad 
    \equilibriumCoefficientLinear\in 
    \Bigl (0
    , \frac{1}{2}\Bigr ], \label{eq:technical2}
  \end{align}
  and that \eqref{eq:monL212}--\eqref{eq:monL2Else} are fulfilled.
Then, there exists a unique $\vectorial{f}^{\vectorial{\varepsilon}}\in C^0([0,+\infty); L_{\textnormal{loc}}^1(\reals)^3) \cap L^{\infty}(\reals_+\times \reals)^3$ weak solution of \eqref{system1}--\eqref{eq:equilibriaForm}--\eqref{f0}.

Moreover, if \eqref{eq:monL212}--\eqref{eq:monL2Else}--\eqref{eq:technical1}--\eqref{eq:technical2} hold while $(\ea, \es)\to(0, 0)$, then $u^{\vectorial{\varepsilon}} \definitionEquality \distributionFunctionLetter_{\zeroVelocitySymbol}^{\vectorial{\varepsilon}}+\distributionFunctionLetter_{\positiveVelocitySymbol}^{\vectorial{\varepsilon}}+\distributionFunctionLetter_{\negativeVelocitySymbol}^{\vectorial{\varepsilon}}$ converges, up to subsequence extraction and for the topology of $C^0([0,+\infty); L_{\textnormal{loc}}^1(\reals)^3)$, to the unique weak entropy solution of \eqref{eq:conservationLaw}.
\end{theorem}
The proof of this result, detailed in \Cref{sec:gettingRidOfBV}, relies on a \strong{regularization of the initial datum}, since it belongs to $L^1_{\textnormal{loc}}(\reals)$.
However, the additional condition \eqref{eq:technical1} must be considered to ensure global existence of $\vectorial{f}^{\vectorial{\varepsilon}}$ with our technique of proof, which is based on the approach by \cite{MR1409928, MR1643175}.
Once global existence is ensured, we can apply the results of \Cref{thm:hugeTheorem} on the solution stemming from the regularized initial datum.

\section{Quasi-monotonicity}\label{sec:quasiMonotonicity}

As highlighted in the introduction, \strong{quasi-monotonicity} plays a central role.
We consider the definition from \cite{MR1409928, MR1643175}.
\begin{definition}[Quasi-monotone function]\label{def:quasiMon}
Let $q>1$, $Q\subset \reals^q$ be an interval with non-empty interior, and $\vectorial{G} : Q\to \reals^q$.
    Then, $\vectorial{G}$ is said to be \strong{quasi-monotone} (on $Q$) if each component $G_i$ is non-decreasing in its $j$-th argument for $j\neq i$.
\end{definition}

Here, we are concerned with the function $\vectorial{G}$ being the right-hand side of \eqref{system1}, that is the ``relaxation'' term.
Let us consider this setting.
   \begin{itemize}
     \item Consider $\equilibriumCoefficientLinear  =\frac{1}{2}$ (\lbmScheme{1}{2} case).
     We note that derivatives of the right-hand side of \eqref{system1} with respect to distribution functions depend on $\vectorial{f}^{\vectorial{\varepsilon}}$ only through $u^{\vectorial{\varepsilon}}$.
    For this reason, the quasi-monotonicity on any given $Q = [\underline{q}_{\zeroVelocitySymbol}, \overline{q}_{\zeroVelocitySymbol}]\times [\underline{q}_{\positiveVelocitySymbol}, \overline{q}_{\positiveVelocitySymbol}] \times [\underline{q}_{\negativeVelocitySymbol}, \overline{q}_{\negativeVelocitySymbol}]\subset  \reals^{3}$ is equivalent to request derivatives of $G_i$ with respect to its $j$-th argument ($j\neq i$) be non-negative when $u^{\vectorial{\varepsilon}}\in [\underline{q}_{\zeroVelocitySymbol}+\underline{q}_{\positiveVelocitySymbol}+\underline{q}_{\negativeVelocitySymbol}, \overline{q}_{\zeroVelocitySymbol}+\overline{q}_{\positiveVelocitySymbol} +\overline{q}_{\negativeVelocitySymbol}] \subset \reals$.
    We thus consider $\underline{q}_{\zeroVelocitySymbol}+\underline{q}_{\positiveVelocitySymbol}+\underline{q}_{\negativeVelocitySymbol}=-\maximumInitialDatum$ and $\overline{q}_{\zeroVelocitySymbol}+\overline{q}_{\positiveVelocitySymbol} +\overline{q}_{\negativeVelocitySymbol} = \maximumInitialDatum$.
    This holds under the subcharacteristic condition 
   \begin{equation}\label{subchar}
   \max_{u\in
   [-\mu_\infty,\mu_\infty]}\frac{|\flux^\prime(u)|}{\lambda}\leq 1.
 \end{equation}
   To see this, observe that the two equations left in \eqref{system1} read
   \begin{equation*}
    \partial_t \distributionFunctionLetter_{\pm}^{\vectorial{\varepsilon}}\pm\latticeVelocity \partial_x \distributionFunctionLetter_{\pm}^{\vectorial{\varepsilon}}=\pm\frac{1}{2\ea}\left( \frac{\flux(u^{\vectorial{\varepsilon}})}{\latticeVelocity}-(\distributionFunctionLetter_{\positiveVelocitySymbol}^{\vectorial{\varepsilon}}-\distributionFunctionLetter_{\negativeVelocitySymbol}^{\vectorial{\varepsilon}})\right).
   \end{equation*}
   The subcharacteristic equation comes from differentiating the right-hand side of the previous expression with respect to $\distributionFunctionLetter_{\mp}^{\vectorial{\varepsilon}}$, and impose they be non-negative.   
   \item Consider $\equilibriumCoefficientLinear  \not=\frac{1}{2}$, we proceed as in the previous case with slightly more involved computations, and request that $1-2\equilibriumCoefficientLinear  \geq 0$ and 
   \begin{equation}\label{cond1}
     \max_{u\in [-\mu_\infty,\mu_\infty]}\frac{|\flux^\prime(u)|}{\lambda}\leq \min\left(  \frac{2  \ea \equilibriumCoefficientLinear }{\es},1-\frac{\ea}{\es}(1-2\equilibriumCoefficientLinear )\right).
 \end{equation}
\begin{itemize}
    \item If $\frac{\ea}{\es}\leq 1$, the condition \eqref{cond1} becomes $\frac{|\flux^\prime(u)|}{\lambda}\leq 
\frac{2  \ea \equilibriumCoefficientLinear }{\es}$. If this holds, we deduce that $\equilibriumCoefficientLinear > 0$, thus $\equilibriumCoefficientLinear \in (0, \tfrac{1}{2}]$. 
\item If $\frac{\ea}{\es}\geq 1$, the condition \eqref{cond1} becomes $\frac{|\flux^\prime(u)|}{\lambda}\leq 1-\frac{\ea}{\es}(1-2\equilibriumCoefficientLinear )$. A necessary condition to have this last inequality is that $1-\frac{\ea}{\es}(1-2\equilibriumCoefficientLinear )> 0$, which entails 
\[
  0<1 \leq \frac{\ea}{\es} < \frac{1}{1-2\equilibriumCoefficientLinear }.
\]
For this to happen, we must have $\equilibriumCoefficientLinear  \in (0, \tfrac{1}{2}]$ as well.
\end{itemize}
 \end{itemize}
We can easily prove the following.
\begin{lemma}
    Let $\equilibriumCoefficientLinear\in(0, \frac{1}{2}]$ and \eqref{subchar} or \eqref{cond1} be fulfilled according to the value of $\equilibriumCoefficientLinear $.
    Then the equilibrium functions are monotone non-decreasing on $[-\mu_\infty,\mu_\infty]$.
\end{lemma}
By virtue of the previous lemma, under $\equilibriumCoefficientLinear\in(0, \frac{1}{2}]$ and \eqref{subchar}--\eqref{cond1}, we can introduce the notations
\begin{equation*}
  \minvec 
  \definitionEquality
  \begin{pmatrix}
    \distributionFunctionLetter_{\zeroVelocitySymbol}^\atEquilibrium(-\maximumInitialDatum)\\
    \distributionFunctionLetter_{\positiveVelocitySymbol}^\atEquilibrium(-\maximumInitialDatum)\\
    \distributionFunctionLetter_{\negativeVelocitySymbol}^\atEquilibrium(-\maximumInitialDatum)
  \end{pmatrix}
  \qquad 
  \text{and}\qquad 
  \maxvec 
  \definitionEquality
  \begin{pmatrix}
    \distributionFunctionLetter_{\zeroVelocitySymbol}^\atEquilibrium(\maximumInitialDatum)\\
    \distributionFunctionLetter_{\positiveVelocitySymbol}^\atEquilibrium(\maximumInitialDatum)\\
    \distributionFunctionLetter_{\negativeVelocitySymbol}^\atEquilibrium(\maximumInitialDatum)
  \end{pmatrix},
\end{equation*}
and especially claim that \Cref{def:quasiMon} holds for a function of the three distribution functions being the right-hand side of \eqref{system1}, taking $Q = [\minvec, \maxvec]$.
This is resumed as follows.
\begin{proposition}[Sufficient quasi-monotonicity conditions for the relaxation system]\label{mono1}
Let 
\begin{equation}\label{L2cond}
    \equilibriumCoefficientLinear \in  ( 0,\tfrac{1}{2}  ]
\end{equation}
and assume that \eqref{subchar} holds when
$\equilibriumCoefficientLinear  = \tfrac{1}{2}$, and \eqref{cond1}
holds when $\equilibriumCoefficientLinear  \neq \tfrac{1}{2}$. Then the right-hand side function in \eqref{system1} is quasi-monotone on
   $[\minvec,\maxvec]$.
\end{proposition}

\section{Qualitative features}\label{sec:qualitativeFeatures}

We now focus on understanding qualitative features of the solutions to \eqref{system1}, in particular concerning the impact of $\ea$, $\es$, and $\equilibriumCoefficientLinear $.
Of particular interest is the possibly different status of $\ea$ compared to that of $\es$.

A first technique of analysis is based on formal \strong{Chapman-Enskog expansions} in the limit of small $\ea$ and $\es$.
A second approach aims at reducing the ``dimensionality'' of the problem by passing from a system of PDEs to one made up of ODEs.
This is done by investigating (at least numerically) solutions under the form of \strong{travelling waves}.
We shall see that these two standpoints are complementary.

In this section, we drop---for the sake of convenience---the superscript $\vectorial{\varepsilon}$ that so far indicated the dependence of the solution of the relaxation scheme on the choice of $(\ea, \es)$.

\subsection{Chapman-Enskog expansion}

The formulation \eqref{eq:systemMoments} is particularly suitable to perform a Chapman-Enskog expansion in the limit of small $\ea$ and $\ea$.
The last two equations entail
\[
  \left\{
    \begin{aligned}
      &v=\flux(u)-\ea(\partial_t v + \lambda^2 \partial_x
      w),\\
      &w=2\equilibriumCoefficientLinear u-\es (\partial_t w +\partial_x v).
    \end{aligned}
  \right.
\]
Hence, the first equation at leading order reads $v=\flux(u)+O(\ea)$.
Taking its time derivative and keeping only the leading order: $\partial_t v =\partial_t \flux(u) +O(\ea) =\flux^\prime(u)\partial_t u +O(\ea) =-\flux^\prime(u)\partial_x v +O(\ea) =-(\flux^\prime(u))^2\partial_x u +O(\ea)$.
Neglecting $O(\ea^2)$ and $O(\ea \es)$ terms:
\[
  v=\flux(u)+\ea \left( (\flux^\prime(u))^2-2\lambda^2
    \equilibriumCoefficientLinear \right)\partial_x u +O(\ea^2) +O(\ea\es),
\]
so that the ``equivalent equation''  obtained from \eqref{eq:systemMoments} is:
\begin{equation}\label{CE}
  \partial_t u + \partial_x \flux(u)=\ea \partial_x \left( \left(2\lambda^2
    \equilibriumCoefficientLinear - (\flux^\prime(u))^2 \right)
  \partial_x u\right) +O(\ea^2) +O(\ea\es),
\end{equation}
where the neglected terms are quadratic in the $\epsilon$'s.
Remark that the diffusion term depends on $\ea$ only, coherently with the lattice Boltzmann setting \cite{aregba2026monotonicity}.
This equation is dissipative if
\begin{equation}\label{CEdissip}
  \equilibriumCoefficientLinear  \geq 0\qquad {\rm and} \qquad 
  \lambda \sqrt{2 \equilibriumCoefficientLinear }\geq
  |\flux^\prime(u)|,
\end{equation}
where $u$ must be thought as belonging to $[-\maximumInitialDatum, \maximumInitialDatum]$.
Note that in the \lbmScheme{1}{2} case
$\equilibriumCoefficientLinear =\frac{1}{2}$, this condition reduces
to the subcharacteristic one \eqref{subchar}.

Back to the general setting, at leading order, only $\ea$ impacts the ``viscous regularisation'' induced by the relaxation process.
The second-order moment $w$ influences this process through its equilibrium, that is $\equilibriumCoefficientLinear $.
The larger this coefficients, the larger the dissipation.

\begin{lemma}[Sufficient quasi-monotonicity conditions imply dissipative Chapman-Enskog equation]\label{lemma:consQuasiMonCR}
  If the assumptions of \Cref{mono1} are fulfilled, then \eqref{CEdissip} is fulfilled as well.
  Moreover, the second inequality in \eqref{CEdissip} is strict if $\equilibriumCoefficientLinear \neq \tfrac{1}{2}$.
\end{lemma}
\begin{proof}
  Condition \eqref{L2cond} implies that the first inequality in \eqref{CEdissip} is valid.
  Furthermore, it also gives that $2\equilibriumCoefficientLinear  \in
  (0, 1]$, and if $\equilibriumCoefficientLinear \neq \tfrac{1}{2}$,
  the interval is also open to the right.  Let $\ea/\es\leq 1$, then \eqref{cond1} gives 
  \begin{equation*}
      \frac{|\flux^\prime(u)|}{\lambda}\leq \frac{2  \ea \equilibriumCoefficientLinear }{\es}
      \leq {  2\equilibriumCoefficientLinear }
      \leq 
      \sqrt{2\equilibriumCoefficientLinear} .
  \end{equation*}
  Let $\ea/\es\geq 1$, then \eqref{cond1} yields
  \begin{equation*}
      \frac{|\flux^\prime(u)|}{\lambda}\leq 1-\frac{\ea}{\es}(1-2\equilibriumCoefficientLinear )
      \leq 2\equilibriumCoefficientLinear \leq \sqrt{ 2\equilibriumCoefficientLinear   }.
  \end{equation*}
\end{proof}
 
\subsection{Travelling wave solutions}

In this section, the relaxation parameters $\ea$ and $\ea$ are given and strictly positive.
We now consider ${u}_- \neq {u}_+$, and look to a profile---moving without deformation at some velocity $c$---connecting the first state at $-\infty$ to the second one at $+\infty$. 
More precisely, we look for travelling wave solutions of \eqref{eq:systemMoments} of the form
\begin{equation*}
    u(t, x) = {u}(x-ct), 
    \qquad 
    v(t, x) = {v}(x-ct), 
    \qquad 
    w(t, x) = {w}(x-ct),
\end{equation*}
such that equilibrium is reached at $\pm\infty$:
\begin{equation}\label{eq:atInfinity}
    \lim_{\xi\to\pm\infty} {u}(\xi) = {u}_{\pm}, 
    \qquad 
    \lim_{\xi\to\pm\infty} {v}(\xi) = \flux({u}_{\pm}), 
    \qquad 
    \lim_{\xi\to\pm\infty} {w}(\xi) = 2\equilibriumCoefficientLinear  {u}_{\pm}.
\end{equation}
Let $\xi = x-ct$. Into \eqref{eq:systemMoments}, this yields:
\begin{equation}\label{eq:systemWaves3}
    \begin{cases}
        -c{u}'(\xi) + {v}'(\xi) = 0, \\
        -c{v}'(\xi) + \lambda^2 {w}'(\xi) = \frac{1}{\ea}\left(
    \flux({u}(\xi))-{v}(\xi)\right), \\
    -c{w}'(\xi) + {v}'(\xi) = \frac{1}{\es}\left(2\equilibriumCoefficientLinear {u}(\xi)-{w}(\xi)\right).
    \end{cases}
\end{equation}
Integrating the first equation on $\reals$ gives $c({u}_+-{u}_-)=\flux({u}_+)-\flux({u}_-)$, which is the Rankine-Hugoniot condition. Hence $c$ is necessarily the velocity of the considered (regularized) shock wave. 

\subsubsection{Previous results in the BGK setting and few explicit solutions}

We now discuss two particular cases in which we are able to explicitly solve \eqref{eq:systemWaves3}. 
This provides a nice connection to existing results in the BGK case and with the previous Chapman-Enskog expansion.
\begin{example}[\lbmScheme{1}{2} scheme]\label{ex:D1Q2}
    Consider $\equilibriumCoefficientLinear  = \tfrac{1}{2}$.
    Inserting the first equation of \eqref{eq:systemWaves3} into the last one, we obtain 
    \begin{equation*}
        c({u}(\xi)-{w}(\xi))'
        =
        \frac{1}{\es} ({u}(\xi)-{w}(\xi)).
    \end{equation*}
    If $c=0$, we obtain that ${u}\equiv {w}$.
    When $c\neq 0$, we deduce that ${u}(\xi)-{w}(\xi) = C e^{\xi/(c\es)}$: imposing equilibrium at $\pm\infty$, we obtain $C=0$, thus that ${u}\equiv {w}$.
    With this, the first two equations out of \eqref{eq:systemWaves3} read
    \begin{equation}\label{eq:tmp1}
    \begin{cases}
        -c{u}'(\xi) + {v}'(\xi) = 0, \\
        -c{v}'(\xi) + \lambda^2 {u}'(\xi)  = \frac{1}{\ea}\left(
    \flux({u}(\xi))-{v}(\xi)\right). 
    \end{cases}
\end{equation}
This is nothing but the system considered by \cite{mascia1996l1nonlinear} in the proof of their Theorem 2.2, which gives sufficient conditions for the existence of the travelling wave above. 
Their condition is that ${u}_- > {u}_+$, the Oleinik condition 
\begin{equation*}
    \flux(s)-\flux({u}_+) - \frac{\flux({u}_+)-\flux({u}_-)}{{u}_+-{u}_-}
    (s-{u}_+)<0,
    \qquad \forall s\in ({u}_+, {u}_-),
\end{equation*}
and that $\latticeVelocity>|c|$.
Equation \eqref{eq:tmp1} can be recast a sole equation on ${u}$.
In the Burgers case $\flux(u) = u^2/2$, we have $c = ({u}_- + {u}_+)/2$, the equation reads 
\begin{equation*}
    {u}'(\xi) = \frac{1}{2\ea(\latticeVelocity^2-c^2)} ({u}(\xi)-{u}_-)({u}(\xi)-{u}_+),
\end{equation*}
which can be solved to yield (selecting it such that ${u}(0) = ({u}_- + {u}_+)/2$):
\begin{equation}\label{eq:hypTangent}
    {u}(\xi)
    =
    \frac{{u}_- + {u}_+}{2}
    +
    \frac{{u}_- - {u}_+}{2}
    \textnormal{tanh}
    \Bigl ( 
    -\frac{{u}_- - {u}_+}{4\ea(\latticeVelocity^2-c^2)}\xi
    \Bigr ).
\end{equation}
Again, this emphasizes that $\es$ plays no role in this case.
Moreover, the solution in \eqref{eq:hypTangent} enjoys the following symmetry property: ${u}(-\xi)-{u}(0)
    =
    {u}(0) - {u}(\xi)$.
\end{example}
\begin{example}[Burgers stationary shock]
Consider the Burgers flux $\flux(u) = u^2/2$, and ${u}_{\mp}=\pm U$ with $U>0$, so that by the Rankine-Hugoniot law,  $c=0$.
 We obtain
  \begin{equation*}
    \begin{cases}
         {v}'(\xi) = 0, \\
       \lambda^2 {w}'(\xi) = \frac{1}{\ea}\left(
    \flux({u}(\xi))-{v}(\xi)\right), \\
    {v}'(\xi) = \frac{1}{\es}\left(2\equilibriumCoefficientLinear {u}(\xi)-{w}(\xi)\right),
    \end{cases}
  \end{equation*}
so
${w}(\xi) = 2\equilibriumCoefficientLinear {u}(\xi)$ and $v=\flux(U)=\flux(-U)=\ud U^2$.
The equation on ${u}$ becomes
\[
  2\lambda^2\equilibriumCoefficientLinear {u}^\prime=\frac{1}{2\ea}({u}^2-U^2)
\]
which can be  explicitly solved (up to translations) to yield the hyperbolic tangent profile
\[
  {u}(\xi)
  = 
  U\frac{1-\textnormal{exp}({\frac{U}{2\ea \latticeVelocity^2\equilibriumCoefficientLinear } \xi})}{1+\textnormal{exp}({\frac{U}{2\ea \latticeVelocity^2\equilibriumCoefficientLinear } \xi})}
  =
  U\,\text{tanh} \Bigl (-\frac{U}{4\ea \latticeVelocity^2\equilibriumCoefficientLinear } \xi \Bigr ) .
\]
This shows that, coherently with the Chapman-Enskog expansion and interpreting numerical diffusion in the relaxation approximation as a term regularizing the shock, the thickness of the regularized profile depends---in this particular setting---only on $\ea$, and not on $\es$.
\end{example}

\subsubsection{Existence of weak shocks}

We adapt Theorem 2.2 by \cite{zumbrun2000existence} on the existence of weak shocks to our scalar setting.
Since we only have \strong{sufficient conditions}, we decided to give them in terms of the previously introduced (sufficient) conditions for quasi-monotonicity.
\begin{theorem}[Consequence of Theorem 2.2 by \cite{zumbrun2000existence}]
Assume that $\equilibriumCoefficientLinear \in (0, \frac{1}{2})$ (the case $\equilibriumCoefficientLinear = \frac{1}{2}$ is thoroughly discussed in \Cref{ex:D1Q2}).
 Let the condition
    \begin{equation*}
    \frac{\flux({u}_+)-\flux({u}_-)}{{u}_+-{u}_-}
    <
    \frac{\flux({u})-\flux({u}_-)}{{u}-{u}_-}
    \qquad 
    \forall
    {u}
    \in (\min({u}_-, {u}_+), \max({u}_-, {u}_+))
\end{equation*}
hold, and that
    \begin{equation*}
        \flux'({u}_-)\neq 0.
    \end{equation*}
Finally, assume that \eqref{cond1} holds with $\maximumInitialDatum = \max(|{u}_-|, |{u}_+|)$.
    Then there exist $\delta_1, \delta_2>0$ such that for $|{u}_-
    -{u}_+|<\delta_1$, a unique (up to a translation), smooth solution of \eqref{eq:systemWaves3} fulfilling \eqref{eq:atInfinity}, such that $|{u}(\xi)-{u}_-|<\delta_2$, $|{v}(\xi)-\flux({u}_-)|<\delta_2$, and $| {w}(\xi)-2\equilibriumCoefficientLinear  {u}_-|<\delta_2$ exists.
\end{theorem}
\begin{proof}[Sketch of the proof]
  Note that assumption (c) in \cite{zumbrun2000existence} requests $|\flux'({u}_-)|\neq \latticeVelocity$, however the quasi-monotonicity assumption, see \Cref{mono1}, prevents from having $|\flux'({u}_-)| =  \latticeVelocity$.
  On the other hand, assumption (c) also requests $\flux'({u}_-)\neq 0$, which we must list in the hypotheses.
  Then, notice that (2.4) and (2.5) in \cite{zumbrun2000existence} hold true since, considering that $\equilibriumCoefficientLinear \neq\tfrac{1}{2}$, our sufficient quasi-monotonicity condition implies that the second inequality in \eqref{CEdissip} is strict, by virtue of \Cref{lemma:consQuasiMonCR}.
  All other assumptions are trivially fulfilled.
\end{proof}

\subsubsection{Equilibria classification and existence of (some) shocks for the Burgers' equation}

Coming back to the general setting, integrating the first equation in \eqref{eq:systemWaves3}, we obtain ${v}(\xi) = c{u}(\xi) + C_v$ and hence
\begin{equation}\label{eq:systemODEuandV}
    \begin{cases}
        -c^2{u}'(\xi) + \lambda^2 {w}'(\xi) = \frac{1}{\ea}\left(
    \flux({u}(\xi))-c{u}(\xi) - C_v\right), \\
    c{u}'(\xi) -c{w}'(\xi)  = \frac{1}{\es}\left(2\equilibriumCoefficientLinear {u}(\xi)-{w}(\xi)\right).
    \end{cases}
\end{equation}
where $C_v = \flux({u}_{\pm}) - c{u}_{\pm} \in\reals$ is a constant.
We now convert the problem to a second-order (if the shock is unsteady: $c\neq 0$) equation on ${u}$ only.
We shall eventually come back to the problem also with $v$ and $w$.
The procedure is based on \cite{bellotti2023truncation}, strongly reminiscent of the computation by \cite{rauch1972l2}---detailed in Appendix \ref{app:equivalentSecondU}.
If we assume (for the moment), that a couple of smooth functions $(u(\xi), w(\xi))$ satisfies \eqref{eq:systemODEuandV}, then $u(\xi)$ alone also fulfills
\begin{equation}\label{eq:secondOrderU}
    c(c^2 - \lambda^2) {u}''(\xi)
    + \Bigl( -\frac{c^2}{\es} - \frac{c^2}{\ea} + \frac{2\lambda^2\equilibriumCoefficientLinear }{\es} + \frac{c}{\ea}\flux'({u}(\xi))\Bigr ){u}'(\xi)
    +\frac{c}{\ea\es}{u}(\xi)
    -\frac{1}{\ea\es}(\flux({u}(\xi))-C_v) = 0.
\end{equation}
Although we were unable to explicitly solve this equation in full generality, even for the Burgers' equation, note that it looks similar to the one of a \strong{damped pendulum}, where the mass can be negative according to the sign of $c$, damping is non-linear, and the restoring force term is non-linear as well.
If we introduce ${z} = {u}'$, this rewrites as 
\begin{equation}\label{eq:shooting}
    \begin{cases}
        {u}'(\xi) = {z}(\xi), \\
         {z}'(\xi) =
    \frac{1}{c(c^2 - \lambda^2)}\Bigl( \frac{c^2}{\es} + \frac{c^2}{\ea} - \frac{2\lambda^2\equilibriumCoefficientLinear }{\es} -\frac{c}{\ea}\flux'({u}(\xi)) \Bigr ){z}(\xi)
    -\frac{1}{\ea\es (c^2 - \lambda^2)}{u}(\xi)
    +\frac{1}{\ea\es c(c^2 - \lambda^2)}(\flux({u}(\xi))-C_v).
    \end{cases}
\end{equation}
Let us now consider the Burgers' flux: from \eqref{eq:secondOrderU}, we obtain the equation 
\begin{equation}\label{eq:BurgersSecondOrderTravellingEq}
     {u}''(\xi)
    + ( \alpha + \beta {u}(\xi) ){u}'(\xi)
    +\gamma\Bigl (-\frac{{u}(\xi)^2}{2 c} + u(\xi)+\frac{C_v}{c} \Bigr ) = 0,
\end{equation}
where 
\begin{equation*}
  \alpha = 
  \frac{1}{c^2 - \lambda^2}
\Bigl ( \frac{2\lambda^2\equilibriumCoefficientLinear }{\es c}   -\frac{c}{\es} - \frac{c}{\ea}\Bigr ),
\qquad 
\beta = 
 \frac{1}{(c^2 - \lambda^2)\ea},
 \qquad 
 \text{and}
 \qquad 
 \gamma =
\frac{1}{(c^2 - \lambda^2)\ea\es}.
\end{equation*}
The system \eqref{eq:shooting} reads 
\begin{equation}\label{es:systemBurgers}
  \begin{cases}
    u'(\xi) = z(\xi),\\
    z'(\xi) = -( \alpha + \beta {u}(\xi) )z(\xi)
    -\gamma\Bigl (-\frac{{u}(\xi)^2}{2 c} + u(\xi)+\frac{C_v}{c} \Bigr ),
  \end{cases}
\end{equation}
hence the Jacobian of its right-hand side:
\begin{equation*}
  \mathscr{J}(u, z)
  =
  \begin{pmatrix}
    0 & 1\\
    -\beta z+\gamma(\frac{u}{c}-1) & -(\alpha + \beta u)
  \end{pmatrix},
  \qquad \text{so}\qquad 
  \mathscr{J}(u, 0)
  =
  \begin{pmatrix}
    0 & 1\\
    \gamma(\frac{u}{c}-1) & -(\alpha + \beta u)
  \end{pmatrix}.
\end{equation*}
Note that the only two equilibria of \eqref{es:systemBurgers} are $(u_{\pm}, 0)$, and we are now looking for a \strong{heteroclinic orbit} connecting them.
We now classify the equilibria by linearization, see Appendix \ref{app:lemma:classificationEquilibriaBurgers} for the proof.
\begin{lemma}[Classification of equilibria by linearization in the Burgers' case]\label{lemma:classificationEquilibriaBurgers}
  Consider the Burgers flux $\flux(u)=u^2/2$.
  Assume that $u_- > u_+$ and $u_+>-u_-$ (thus $c>0$).
  Moreover, select $\latticeVelocity>0$ such that $\latticeVelocity>c$.
  Then 
  \begin{itemize}
    \item The equilibrium $(u_+, 0)$ of \eqref{es:systemBurgers} is a \strong{saddle point}. Hence, the eigenvalues of the Jacobian are real with opposite sign, and thus the system does not oscillate in the vicinity of this equilibrium.
    \item The equilibrium $(u_-, 0)$ of \eqref{es:systemBurgers}, 
    \begin{itemize}
      \item If $\equilibriumCoefficientLinear  <
    \frac{c}{2\lambda^2 }
    \Bigl ( 
    \frac{u_- + u_+}{2} -\frac{u_- - u_+}{2}\frac{\es}{\ea}
    \Bigr )$, is a \strong{sink}, hence the heteroclinic orbit does not exist.
      \item If $\equilibriumCoefficientLinear  =
    \frac{c}{2\lambda^2 }
    \Bigl ( 
    \frac{u_- + u_+}{2} -\frac{u_- - u_+}{2}\frac{\es}{\ea}
    \Bigr )$, is \strong{not hyperbolic}, and we cannot conclude.
    \item If $\equilibriumCoefficientLinear  >
    \frac{c}{2\lambda^2 }
    \Bigl ( 
    \frac{u_- + u_+}{2} -\frac{u_- - u_+}{2}\frac{\es}{\ea}
    \Bigr )$, is a \strong{source}, and the heteroclinic orbit can exist.
      More precisely.
      \begin{itemize}
        \item If $\equilibriumCoefficientLinear 
 < 
 \frac{c}{2\lambda^2}
 \Bigl ( \frac{u_- + u_+}{2} - \frac{u_- - u_+}{2}\frac{\es}{\ea} + 2 \sqrt{(\lambda^2-c^2)
  \frac{\es}{\ea}\frac{u_- - u_+}{u_-+u_+}
}\Bigr )$, we have an \strong{unstable focus}, and the system can oscillate in the vicinity of the equilibrium.
\item If $\equilibriumCoefficientLinear 
 = 
 \frac{c}{2\lambda^2}
 \Bigl ( \frac{u_- + u_+}{2} - \frac{u_- - u_+}{2}\frac{\es}{\ea} + 2 \sqrt{(\lambda^2-c^2)
  \frac{\es}{\ea}\frac{u_- - u_+}{u_-+u_+}
}\Bigr )$, we have an \strong{unstable degenerate node}.
\item If $\equilibriumCoefficientLinear 
 > 
 \frac{c}{2\lambda^2}
 \Bigl ( \frac{u_- + u_+}{2} - \frac{u_- - u_+}{2}\frac{\es}{\ea} + 2 \sqrt{(\lambda^2-c^2)
  \frac{\es}{\ea}\frac{u_- - u_+}{u_-+u_+}
}\Bigr )$, we have an \strong{unstable node}.
      \end{itemize}
    \end{itemize}
  \end{itemize}
\end{lemma}
In what follows, we shall perform numerical illustrations with $u_- = 1$ and $u_+ = 0$ for the Burgers' equation.
We thus give sufficient conditions for the existence of a travelling wave in this framework.
The proof is based on the dissipation of an energy, which ensures trapping of the trajectories.
\begin{proposition}[Sufficient conditions for the existence of a $u_-=1\to u_+ = 0$ travelling wave in the Burgers' case]\label{prop:SufficientExistenceTWBurgers}
  Consider the Burgers flux $\flux(u)=u^2/2$.
  Take  $u_- = 1$ and $u_+ = 0$ (thus $c=\tfrac{1}{2}$).
  Moreover, select $\latticeVelocity>0$ such that $\latticeVelocity>c = \tfrac{1}{2}$.
  If 
  \begin{equation}\label{eq:enoughFriction}
  \equilibriumCoefficientLinear
  >
  \frac{1}{8\latticeVelocity^2}
  \Bigl ( 1+\frac{\es}{\ea}\Bigr ),
\end{equation}
there exists a unique (up to translations) smooth solution of \eqref{eq:systemWaves3} fulfilling \eqref{eq:atInfinity}. 
Moreover, this solution satisfies the bounds
  \begin{equation}\label{eq:boundsTravellingWave}
    0\leq u(\xi)< \frac{3}{2}
    \qquad 
    \text{and}
    \qquad 
    |u'(\xi)|
    <
    \frac{1}{\sqrt{(\lambda ^2-c^2)\ea\es}}u(\xi)\sqrt{1-\frac{2}{3}u(\xi)}
    \leq 
    \frac{1}{\sqrt{3(c^2 - \lambda^2)\ea\es}}.
  \end{equation}
\end{proposition}
  Notice that the first bound in \eqref{eq:boundsTravellingWave} allows overshoots above $u_- = 1$ (but not undershoots below $u_+ = 0$): the heteroclinic profile $u(\xi)$ is indeed \strong{not necessarily monotone}, especially in the vicinity of $u_- = 1$ and in the case where this equilibrium is an unstable focus, see \Cref{lemma:classificationEquilibriaBurgers}.
  Notice that the right-hand side of \eqref{eq:enoughFriction} rapidly goes to zero as $\lambda$ increases, hence the condition is generally not far from $\equilibriumCoefficientLinear>0$.
  Moreover, there is no upper bound on $\equilibriumCoefficientLinear$.
\begin{proof}[Proof of \Cref{prop:SufficientExistenceTWBurgers}]
  The proof starts by considering \eqref{eq:BurgersSecondOrderTravellingEq} (for $u$) or \eqref{es:systemBurgers} (for $u$ and $z$), prove that they admit a unique smooth heteroclinic orbit connecting the equilibria, and then construct from this $v$ and $w$, solutions of \eqref{eq:systemWaves3} with \eqref{eq:atInfinity}.
  To prove the existence of the heteroclinic, after having reversed the ``time-arrow'' to use $u_- = 1$ as a stable equilibrium, we argue in two stages.
  \begin{enumerate}
    \item First, by the (un)stable manifold theorem \cite[Section 2.7]{perko2013differential}, we prove that we can select a point close to $(u_+, 0) = (0, 0)$ such that the backward trajectory starting from it reaches $(u_+, 0) = (0, 0)$ at $-\infty$.
    \item Second, we prove that the forward trajectory starting from the point selected at the previous stage remains trapped in an energy-well.
    We conclude that it reaches $(u_-, 0) = (1, 0)$ at $+\infty$ using the Lasalle's invariance principle, \confer{} \cite[Section 9.2]{hirsch2013differential}.
  \end{enumerate}
  Before detailing the demonstration, we note that \eqref{eq:boundsTravellingWave} naturally emerges from the idea of proof about existence, thanks to the energy trapping.
  Let us start.

  Assume that $\latticeVelocity>c$.
We thus have $\beta<0$ and $\gamma<0$, while the sign of $\alpha$ is dubious.
According to \Cref{lemma:classificationEquilibriaBurgers}, we consider the necessary condition, based on the linearization,  to have a heteroclinic orbit, which now reads 
\begin{equation}\label{eq:tmp6}
  \equilibriumCoefficientLinear
  >
  \frac{1}{8\latticeVelocity^2}
  \Bigl ( 1-\frac{\es}{\ea}\Bigr ).
\end{equation}
Note that the right-hand side is not necessarily positive.
Now imagine that $u\in [0, 1]$: we thus would like that the damping coefficient $\alpha + \beta u$ has a sign. We thus request $\alpha<0$.
This condition is 
\begin{equation*}
  \equilibriumCoefficientLinear
  >
  \frac{1}{8\latticeVelocity^2}
  \Bigl ( 1+\frac{\es}{\ea}\Bigr ),
\end{equation*}
with a right-hand side which is always larger than the one of \eqref{eq:tmp6}: the condition is more restrictive.
Now that $\alpha<0$, the damping coefficient has the ``wrong'' sign whenever $u\in[0, 1]$.
For this reason, we reverse the coordinate $\xi$: let $\tau = -\xi$, so that $\lim_{\xi\to\pm\infty}\tau = \mp \infty$, and consider $x(\tau) = u(\xi)$.
Equation \eqref{eq:BurgersSecondOrderTravellingEq} becomes
\begin{equation}\label{eq:BurgersSecondOrderTravellingEqReversedTime}
     {x}''(\tau)
    - ( \alpha + \beta {x}(\tau) ){x}'(\tau)
    +\gamma x(\tau)  (1-{x}(\tau)  ) = 0,
\end{equation}
or equivalently 
\begin{equation}\label{eq:systemPendulumReversed}
  \begin{cases}
    x'(\tau)
  =
  y(\tau), \\
  y'(\tau)
  =
  ( \alpha + \beta {x}(\tau) )y(\tau)
  -\gamma x(\tau)  (1-{x}(\tau)  ).
  \end{cases}
\end{equation}
The change of variable has kept the same equilibria and changed their stability: $(x, y) = (1, 0)$ (at $\tau = +\infty$, or $\xi = -\infty$) is now stable, and we exploit the instability of $(x, y) = (0, 0)$ (at $\tau = -\infty$ or $\xi = +\infty$), which is again a saddle point.
\begin{enumerate}
  \item The point $(x, y) = (0, 0)$ is an equilibrium for \eqref{eq:systemPendulumReversed} and, under the current assumptions, the Jacobian of the right-hand side (denoted $\mathscr{J}$) has one eigenvalue with negative real part (stable) and the other one with positive real part (unstable). In particular 
  \begin{equation*}
  \mathscr{J}(x = 0, y = 0)
  =
  \begin{pmatrix}
    0 & 1\\
    -\gamma & \alpha
  \end{pmatrix},
\end{equation*}
hence we are interested in the unstable eigenvalue 
\begin{equation*}
  \lambda_{u}
  =
  \frac{1}{2}
  (\alpha + \sqrt{\alpha^2-4\gamma})>0, 
  \qquad \text{with associated eigenvector}
  \qquad 
  \begin{pmatrix}
    1 \\
    \lambda_{u}
  \end{pmatrix}.
\end{equation*}
By the (un)stable manifold theorem, there exists a 1-dimensional differentiable manifold $\mathscr{U}$, tangent to $E^{u} = \text{span}(\transpose{(1, \lambda_u)})$ at $(x, y) = (0, 0)$, such that for all $(x_0, y_0)\in\mathscr{U}$, we have
\begin{equation*}
  \vectorial{\phi}_{\tau}(x_0, y_0)\in\mathscr{U}, 
  \quad \forall\tau\leq 0 
  \qquad \text{and}\qquad 
  \lim_{\tau\to-\infty} \vectorial{\phi}_{\tau}(x_0, y_0) = (0, 0),
\end{equation*}
where $\vectorial{\phi}_{\tau}(x_0, y_0)$ denotes the flow of  \eqref{eq:systemPendulumReversed} with initial datum $(x_0, y_0)$.

We now consider $(x_0, y_0)\in\mathscr{U}$ close enough to $(x, y) = (0, 0)$ so that it dwells arbitrarily close to $E^{u}$.
The (un)stable manifold theorem ensures $\lim_{\tau\to-\infty} \vectorial{\phi}_{\tau}(x_0, y_0) = (0, 0)$.
The point $(x_0, y_0)$ is our candidate to have $\{\vectorial{\phi}_{\tau}(x_0, y_0)\, :\, \tau\leq 0\}\cup \{\vectorial{\phi}_{\tau}(x_0, y_0)\, :\, \tau\geq 0\}$ be a sought heteroclinic.
Looking at the vector spanning $E^{u}$, we  have that either $(x_0, y_0)$ is in the first quadrant, or it belongs to the third one.
Let us now show that the latter case is incompatible with $\{\vectorial{\phi}_{\tau}(x_0, y_0)\, :\, \tau\leq 0\}\cup \{\vectorial{\phi}_{\tau}(x_0, y_0)\, :\, \tau\geq 0\}$ being a heteroclinic, which by the way shows that---if this particular kind of trajectory exists---it is also unique.
\begin{equation*}
  \text{Let}
  \quad
  x = 0,
  \quad \text{then from \eqref{eq:systemPendulumReversed}}\quad
    x'
  =
  y \quad \text{and}\quad 
  y'
  =
  \alpha y,
\end{equation*}
so for $y<0$, the vector field points inwards the third quadrant.
\begin{equation}\label{eq:vectorField}
  \text{Let}
  \quad
  y = 0,
  \quad \text{then from \eqref{eq:systemPendulumReversed}}\quad
    x'
  =
  0 \quad \text{and}\quad 
  y'
  =
  -\gamma x(1-x),
\end{equation}
so for $x<0$, the vector field points (vertically) inwards the third quadrant.
Therefore, if $(x_0, y_0)$ belongs to the third quadrant, $\{\vectorial{\phi}_{\tau}(x_0, y_0)\, :\, \tau\geq 0\}$ cannot leave it, so in particular there is no chance that $\vectorial{\phi}_{\tau}(x_0, y_0)\to (1, 0)$ as $\tau\to+\infty$.

We therefore fix $(x_0, y_0)\in\mathscr{U}$ with $x_0>0$ and $y_0>0$ with which $\lim_{\tau\to-\infty} \vectorial{\phi}_{\tau}(x_0, y_0) = (0, 0)$.

\item In a standard fashion, let us multiply \eqref{eq:BurgersSecondOrderTravellingEqReversedTime} by $x'=y$, and thus obtain
\begin{equation}\label{eq:energyBalance}
  \frac{\differential}{\differential\tau}
  \underbrace{\Bigl ( 
  \overbrace{\tfrac{1}{2} y(\tau)^2}^{\text{kin. en.}}
  + \overbrace{\gamma \bigl ( -\tfrac{1}{3}x(\tau)^3 + \tfrac{1}{2}x(\tau)^2\bigr )}^{\text{potential energy}}
  \Bigr )}_{=:\mathscr{E}(x(\tau), y(\tau))}
  =
  ( \alpha + \beta {x}(\tau) )(y(\tau))^2.
\end{equation}
The function $\mathscr{E}: \mathscr{O}\to \reals$, where
$\mathscr{O}\definitionEquality \{(x,
y)\,:\,x>-\frac{\alpha}{\beta}\}$ is open, is clearly
differentiable, and $(1, 0)\in \mathscr{O}$. Clearly, $\mathscr{E}$ reaches its minimum for $x\geq 0$ at the unique point $(1,0)$
and $\mathscr{E}(1,0)=\frac{\gamma}{6}<0$. Moreover 
\[
  \dot{\mathscr{E}}(x(\tau), y(\tau)) \leq 0\quad \text{for}\quad
  (x(\tau), y(\tau))\in\mathscr{O}\smallsetminus (1, 0).
\]
Therefore, $\mathscr{E}: \mathscr{O}\to \reals$ is a Lyapunov function for the equilibrium $(1, 0)$.

Also remark that $\mathscr{E}(0, 0) = 0$.
Note that thanks to $\alpha<0$, we automatically have that
\begin{equation}\label{eq:conditionPentes}
  \frac{1}{2}
  (\alpha + \sqrt{\alpha^2-4\gamma})
  <
  \sqrt{-\gamma}.
\end{equation}
Thus, up to pick once again $(x_0, y_0)$ on the unstable manifold and as close to $(0, 0)$ as needed, we have that 
\begin{equation*}
  (x_0, y_0)
  \in 
  \mathring{\mathscr{D}},
  \qquad \text{where}\qquad 
  \mathscr{D}
  \definitionEquality \{  (x,y)\in [0,+\infty)\times \reals \; : \;
  \mathscr{E}(x,y)\leq 0\}.
\end{equation*}
This fact and the origin of \eqref{eq:conditionPentes} come from the fact that the ``droplet'' closed and bounded set $\mathscr{D}$ admits the explicit parametrization
\begin{equation*}
    \mathscr{D}
    =
    \Biggl \{
    (x,y)\in \Bigl [0, \frac{3}{2}\Bigr ]\times \reals \;  : \; |y| \leq x  \sqrt{-\gamma\Bigl (1-\frac{2}{3}x \Bigr )}
    \Biggr \},
  \end{equation*}
and $\sqrt{-\gamma}$ is nothing but the slope of one of the tangents to $\partial\mathscr{D}$ at $(0, 0)$.
Since $(x_0, y_0)
  \in 
 \mathring{\mathscr{D}}$ and by \eqref{eq:energyBalance}, energy cannot increase on the forward trajectory from $(x_0, y_0)$, thus this trajectory cannot escape $\mathring{\mathscr{D}}$:
  \begin{equation*}
    \{\vectorial{\phi}_{\tau}(x_0, y_0)\, :\, \tau\geq 0\}
    \subset 
   \mathring{\mathscr{D}}.
  \end{equation*}
  Note that $\mathring{\mathscr{D}}$ is an open set contained in $\mathscr{O}$, and contains $(1, 0)$.
  Thus, $\mathscr{E}:\mathring{\mathscr{D}}\to \reals$ is a Lyapunov function.
  We continue by defining the closed set
  \begin{equation*}
    \mathscr{P}
    \definitionEquality
    \{(x,y)\in \mathscr{D} \; : \; \mathscr{E}(x, y)\leq \mathscr{E}(x_0, y_0)\}
    \subset\mathring{\mathscr{D}},
  \end{equation*}
  with $ \mathscr{E}(x_0, y_0)<0$, which contains $(1, 0)$ too.
  By the same way of reasoning through energy, $\mathscr{P}$ is positively invariant.
  Using \eqref{eq:energyBalance} and \eqref{eq:vectorField}, we see that there is no entire solution in $\mathscr{P}\smallsetminus (1, 0)$ on which $\mathscr{E}$ is constant (this is the aim of having considered $\mathscr{P}$ instead of ${\mathscr{D}}$, in order to exclude $(0, 0)$).
  We thus conclude by the Lasalle's invariance principle (\cite[Section 9.2]{hirsch2013differential}) that $(1, 0)$ is asymptotically stable and that $\mathscr{P}$ is contained in the basin of attraction of $(1, 0)$, therefore 
  \begin{equation*}
    \lim_{\tau\to+\infty}
    \vectorial{\phi}_{\tau}(x_0, y_0) = (1, 0).
  \end{equation*}

\end{enumerate}

All these arguments give the existence of the heteroclinic. Uniqueness has been previously discussed.
  To check that we have solved the original problem of existence of a travelling wave profile for the relaxation system, we can compute $v(\xi)$ from $u(\xi)$ via $v(\xi) 
    =
    c u(\xi)$, and easily check that $v(\pm\infty) = \flux(u(\pm\infty))$.
  Looking at \eqref{eq:systemODEuandV}, the first equation gives the derivative of $w$ as function of $u$ and $u'$, previously constructed as unique heteroclinic for \eqref{eq:secondOrderU}:
  \begin{equation*}
      {w}'(\xi) =\frac{1}{\lambda^2}\Bigl ( \frac{1}{4}{u}'(\xi) + \frac{{u}(\xi)}{2\ea}\left(
    {u}(\xi) -1\right)\Bigr ),
  \end{equation*}
  which also gives $w'(\pm\infty) = 0$.
  Into the second equation, we find the expression for $w$:
  \begin{equation*}
    {w}(\xi) = 2\equilibriumCoefficientLinear {u}(\xi) - \frac{\es}{2}{u}'(\xi) +\frac{\es}{2\lambda^2}\Bigl ( \frac{1}{4}{u}'(\xi) + \frac{{u}(\xi)}{2\ea}\left(
    {u}(\xi) -1\right)\Bigr ),
  \end{equation*}
  which by the way gives $w(\pm\infty)
    =
    2\equilibriumCoefficientLinear u(\pm\infty)$.
  By construction, this smooth solution $(u(\xi), w(\xi))$ satisfies \eqref{eq:systemODEuandV}.
  It is unique since if another solution of \eqref{eq:systemODEuandV} exists, as previously explained, it also fulfills \eqref{eq:secondOrderU}, which admits a unique heteroclinic as previously demonstrated.

\end{proof}

\subsubsection{Profile for the truncated Chapman-Enskog expansion}

We now consider---for the sake of comparison with actual (numerically obtained) travelling wave profiles for \eqref{eq:systemMoments} (in particular for $u$)---the truncated version of \eqref{CE}.
It can be formally rewritten as
\begin{equation}\label{eq:truncatedCE}
   \partial_t u + \partial_x \flux(u)=\ea \partial_x \left( \left(2\lambda^2
    \equilibriumCoefficientLinear - (\flux^\prime(u))^2 \right)
  \partial_x u\right) .
   \end{equation}
Looking for travelling wave solutions at velocity $c$, we obtain
\[
  -cu^\prime + \left(\flux(u)\right)^\prime=\ea\bigl( (2\lambda^2
  \equilibriumCoefficientLinear -(\flux'(u))^2) u^\prime\bigr)^\prime
  ,
\]
which we integrate once to find
\begin{equation}\label{eq:solTravellingCE}
    u'(\xi)
    =
    \frac{1}{\ea\left(2\lambda^2
    \equilibriumCoefficientLinear -\flux'({u}(\xi))^2\right)}
    \left(
    \flux({u}(\xi))
    -c{u}(\xi)
    -C_v
    \right),
\end{equation}
where $C_v = \flux({u}_{\pm})- c{u}_{\pm}   $.
Of course the heteroclinic orbit fulfilling \eqref{eq:solTravellingCE} can exists only if dissipation is always strictly positive, \confer{} the denominator in the previous expression.

\subsubsection{Numerical results}

\begin{figure}[h]
    \centering
    \includegraphics[width=1\textwidth]{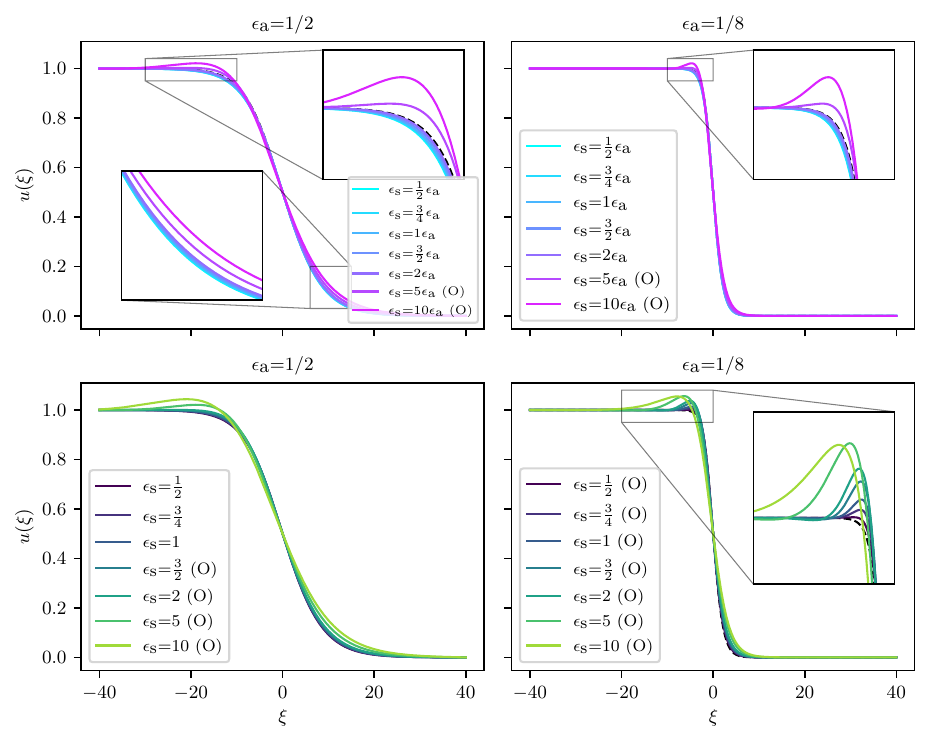}
    \caption{Approximate ${u}$ for $\equilibriumCoefficientLinear=\tfrac{1}{4}$ for different values of relaxation times.
    The black dashed line depicts the travelling wave profile associated with the truncated Chapman-Enskog equation. The mention (O) in the legend indicates oscillation close to the equilibrium $u_- = 1$ according to \Cref{lemma:classificationEquilibriaBurgers}, \idEst{} non-zero imaginary parts in the unstable eigenvalues.}
    \label{fig:steadyTRT-impact-eps-s}
\end{figure}

\begin{figure}[h]
    \centering
    \includegraphics[width=1\textwidth]{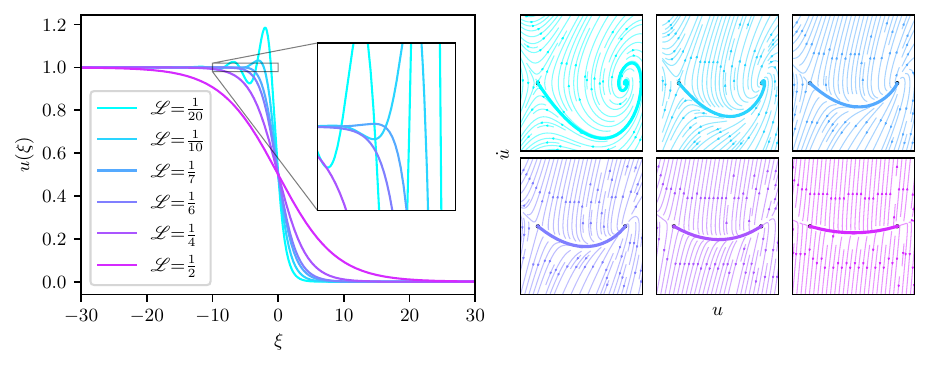}
    \caption{Left: approximate ${u}$ for $\ea=\es = \tfrac{1}{4}$ for different values of $\equilibriumCoefficientLinear$. Right: corresponding phase portraits.}
    \label{fig:steadyTRT-impact-L2}
\end{figure}

\begin{figure}[h]
    \centering
    \includegraphics[width=1\textwidth]{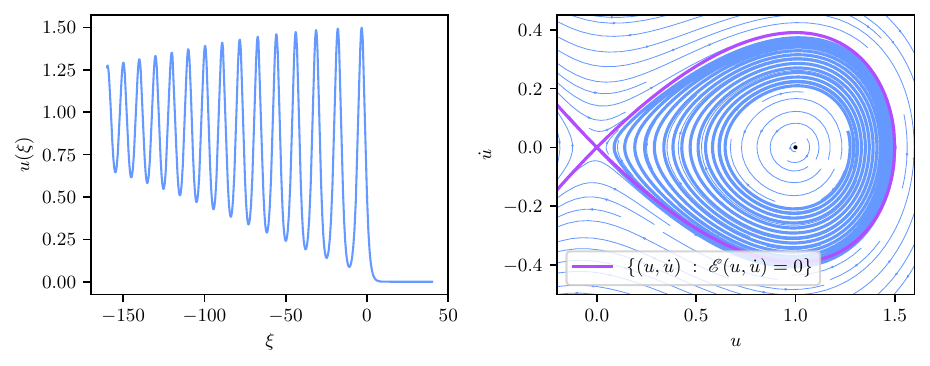}
    \caption{Left: Approximate ${u}$ for $\ea=1$, $\es = \tfrac{1}{4}$, and $\equilibriumCoefficientLinear=\tfrac{1}{87}<\tfrac{1}{8\latticeVelocity^2} (1+\tfrac{\es}{\ea}) = \frac{5}{288}$. Right: corresponding phase portrait and zero-energy isoline, with the energy given in \eqref{eq:energyBalance}.
    Note that zero energy corresponds to the equilibrium $(0, 0)$.}
    \label{fig:steadyTRT-sufficient}
\end{figure}

In order to study the impact of the three parameters $\ea$, $\es$, and $\equilibriumCoefficientLinear $ on the shape of travelling waves, we have implemented a \strong{numerical procedure} to approximate the profile ${u}$.
To this end, we consider \eqref{eq:secondOrderU} to find $u$, and then $v$ and $w$ could be found afterwards.
The numerical procedure is divided into two steps.\footnote{We have also tried to solve with a shooting method for \eqref{eq:shooting} for $\xi\geq 0$, starting from ${u}(0)=\tfrac{1}{2}({u}_- + {u}_+)$ and adjusting ${u}'(0)$ to obtain the heteroclinic connection. We then integrate backwards for $\xi<0$. However, even with a Radau IIA solver of order five \cite{hairer1999stiff}, this procedure is extremely sensitive---let us say ``stiff''---to the slope ${u'}(0)$ and poorly performs on the wide range of parameters that we want to consider.}

\begin{enumerate}
  \item Solve a boundary value problem associated with \eqref{eq:secondOrderU} for $\xi \in [0, L]$ with $L$ large enough\footnote{In all simulations below, we took $L=40$.}. Boundary conditions are 
  \begin{equation*}
    {u}(0) = 
    \frac{{u}_- + {u}_+}{2}
    \qquad \text{and}
    \qquad {u}(L) = {u}_{+}.
  \end{equation*} 
  This gives an approximation of ${u}(\xi)$ for $\xi \geq 0$, which is obtained using the \texttt{solve\textunderscore{}bvp} command in the \texttt{scipy} package.
  \item Solve the initial value problem associated with \eqref{eq:shooting} for $\xi < 0$, using initial datum
  \begin{equation*}
    {u}(0) = 
    \frac{{u}_- + {u}_+}{2}
    \qquad \text{and}
    \qquad 
    {u}'(0) = 
    \lim_{\xi\to 0^+}{u}'(\xi),
  \end{equation*}
  where $\lim_{\xi\to 0^+}{u}'(\xi)$ is approximated with the solution computed at the previous step for $\xi \geq 0$, using a first-order finite difference.
  The initial value problem is approximated with a Radau IIA solver of order five \cite{hairer1999stiff}, implemented within the \texttt{solve\textunderscore{}ivp} command.
\end{enumerate}

As previously said, we consider the states at infinity ${u}_- = 1$ and $u_+ = 0$, and fix the kinetic velocity to $\latticeVelocity = 3$.
We consider a Burgers' flux, so the velocity of the shock is $c=\tfrac{1}{2}$ and the Courant number equals
\begin{equation}\label{tmp4}
  \sup_{u}\frac{|\flux^\prime(u)|}{\lambda} = \frac{1}{3}.
\end{equation}
According to \Cref{rem:onTherationOfpar}, and in particular \eqref{eq:minimalL}, we have a non-empty area in the $\vectorial{\varepsilon}$ plane that ensures quasi-monotonicity, which entails dissipation in the Chapman-Enskog expansion (cf. \Cref{lemma:consQuasiMonCR}), upon having
\begin{equation}\label{tmp5}
  \frac{1}{6}\leq \equilibriumCoefficientLinear\leq \frac{1}{2}.
\end{equation}

\begin{itemize}
    \item We first fix $\equilibriumCoefficientLinear  = \tfrac{1}{4}$ and change the relaxation times.
    Through \Cref{rem:onTherationOfpar}, we fulfill \eqref{eq:monL2Else} whenever 
    \begin{equation}\label{tmp3}
      \frac{3}{4}\ea\leq \es\leq \frac{3}{2}\ea.
    \end{equation}

    Results are presented in \Cref{fig:steadyTRT-impact-eps-s}, and we notice the following facts.
    \begin{itemize}
      \item The thickness of the transition is \strong{essentially determined by the value of $\ea$}. Indeed, the smaller $\ea$, the thinner the transition layer.
      \item The value of $\es$ \strong{contributes to a lesser extent}.
      \item In the upper row, three choices (second, third, and fourth) comply with \eqref{tmp3}. 
      In the lower row, the first two values on the left plot meet  \eqref{tmp3}, and none on the right one.
      When this constraint is violated with $\es>\ea$ by a lot, we see that the solution dwells outside $[0, 1]$ behind the transition layer, and thus is non-monotone.
      This fact questions the fact that the Courant number be given by \eqref{tmp4} in these circumstances.

      When \eqref{tmp3} is fulfilled, discrepancies with respect to the travelling wave solution of the truncated Chapman-Enskog expansion, \confer{} \eqref{eq:solTravellingCE}, are moderate, and the profile ${u}(\xi)$ strongly resembles to 
      \begin{equation*}
        \frac{{u}_- +{u}_+}{2}
        +
        \frac{{u}_- -{u}_+}{2}
        \textnormal{tanh}
        \Bigl ( 
          -\frac{{u}_- -{u}_+}{8\ea\latticeVelocity^2\equilibriumCoefficientLinear}
          \xi
          \Bigr 
        ).
      \end{equation*}

      Finally, we notice that linearizations close to $u_- = 1$, see \Cref{lemma:classificationEquilibriaBurgers} convey a precise description of oscillations (thus lack of monotonicity of the profile).
    \end{itemize}

    \item We then fix $\ea =\es = \frac{1}{4}$, and vary $\equilibriumCoefficientLinear $, sometimes violating \eqref{tmp5}.
    Results are provided in \Cref{fig:steadyTRT-impact-L2}.
    As for $\ea$, we see that the value of $\equilibriumCoefficientLinear $ \strong{modifies the thickness of the transition}. 
    This is compatible with the Chapman-Enskog expansion, which suggest that larger $\equilibriumCoefficientLinear $ increase the dissipation of the relaxation model.
    The profiles become non-monotone and no longer stay in $[0, 1]$ when \eqref{tmp5} is not fulfilled. 
    In particular, oscillations behind the transition become larger and larger as $\equilibriumCoefficientLinear\to 0$.

    Notice that the value of $\equilibriumCoefficientLinear$ that determines between oscillation or not, see \Cref{lemma:classificationEquilibriaBurgers}, is equal in these conditions to 
    \begin{equation*}
      \frac{\sqrt{35}}{36} 
      \in (\tfrac{1}{7}, \tfrac{1}{6}),
    \end{equation*}
    see \Cref{lemma:classificationEquilibriaBurgers}.
    This is indeed precisely reproduced on the simulation, and hence indicates that linearizations provide an accurate description of the behavior of the travelling wave.
    For all the considered values of $\equilibriumCoefficientLinear$, since they are larger than $\tfrac{1}{36}$ (the right-hand side of \eqref{eq:enoughFriction}), the sufficient conditions from \Cref{prop:SufficientExistenceTWBurgers} for the existence of a travelling wave are met.
    Note that the solution does not grow above $\frac{3}{2}$, accordingly to \eqref{eq:boundsTravellingWave}.
    \item Finally, we considered a quite extreme case where $\ea=1$, $\ea = \tfrac{1}{4}$, and $\equilibriumCoefficientLinear = \tfrac{1}{87}$, which is (moderately) below the value by \eqref{eq:enoughFriction}: a sufficient condition for the existence of the travelling wave.
    The result in \Cref{fig:steadyTRT-sufficient} seems to empirically confirm that \eqref{eq:enoughFriction} is not necessary.
    Indeed, since $\alpha>0$, the energy $\mathscr{E}$ in \eqref{eq:energyBalance} grows positive in the vicinity of the unstable equilibrium, but the dynamics is still attracted back to the droplet $\mathscr{D}$, where it eventually dwells. 
    Very interestingly, the trajectory passes quite close to the bounds by \eqref{eq:boundsTravellingWave} that are deduced by the zero-energy isoline.
\end{itemize}

   \section{Numerical scheme for the relaxation system: fixed $(\ea, \es)$}\label{sec:numericalScheme}
   
   We now introduce a numerical scheme for \eqref{system1} under a specific form.
   The aim is not to approximate the solution of \eqref{system1} \emph{per se}, but rather to use the numerical scheme at convergence to deduce useful properties on the solution of \eqref{system1}.
  Remark that $(\ea, \es)$ are now fixed.

   \subsection{Algorithm}
   
   Let $\spaceStep>0$ be the uniform space step of the spatial grid.
   We select the step of the time grid, denoted $\timeStep$, such that $\lambda\frac{\Delta t}{\Delta
      x}=1$, where $\lambda>0$ is the velocity the appears on the left-hand side of \eqref{system1}.
     The scheme we consider is inspired by lattice Boltzmann methods, in the sense that thanks to the previous equality, discretize the left-hand side of \eqref{system1} by upwind schemes yields shifts on the grid.
      This is done because of the very simple structure it provides.
      Still, one could have employed any monotone scheme for the linear transport equation.
      However, contrarily to lattice Boltzmann schemes, we solve the right-hand side of \eqref{system1} exactly.
      
      After an initialization of the data at equilibrium:
      \begin{equation}
      \discrete{f}^{\vectorial{\varepsilon}, 0}_{i, j}=\distributionFunctionLetter_i^{\atEquilibrium}\Bigl ( \frac{1}{\spaceStep} \int_{ (j-\frac{1}{2})\spaceStep}^{ (j+\frac{1}{2})\spaceStep} u^\initial(\spaceVariable)\differential{\spaceVariable} \Bigr), 
      \qquad j \in\relatives,
      \quad \indexVelocity=\zeroVelocitySymbol,\positiveVelocitySymbol, \negativeVelocitySymbol,
  \end{equation}
     the algorithms reads as follows.
     Let $\indexTime\in\naturals$.
\begin{itemize}
  \item Relaxation. Set $\discrete{u}^{\vectorial{\varepsilon}, n}_{j} = \discrete{f}^{\vectorial{\varepsilon}, n}_{\zeroVelocitySymbol,j} + \discrete{f}^{\vectorial{\varepsilon}, n}_{\positiveVelocitySymbol,j} + \discrete{f}^{\vectorial{\varepsilon}, n}_{\negativeVelocitySymbol,j}$ and
    \begin{equation}\label{collision}
    \left\{
      \begin{aligned}
        &
        \discrete{f}^{\vectorial{\varepsilon},n, \collided}_{\zeroVelocitySymbol,j}=\left(1-\expo^{-\frac{\dt}{\es}}\right)\distributionFunctionLetter^{\atEquilibrium}_{\zeroVelocitySymbol}(\discrete{u}^{\vectorial{\varepsilon}, n}_j)+\expo^{-\frac{\dt}{\es}}
        \discrete{f}^{\vectorial{\varepsilon}, n}_{\zeroVelocitySymbol,j},\\
          &   \discrete{f}^{\vectorial{\varepsilon},n,\collided}_{\pm,j}=\ud\left( \expo^{-\frac{\dt}{\es}}+
            \expo^{-\frac{\dt}{\ea}} \right)\discrete{f}^{\vectorial{\varepsilon}, n}_{\pm,j}+\left(
            1-\ud\left(\expo^{-\frac{\dt}{\es}}+\expo^{-\frac{\dt}{\ea}}\right)\right)\distributionFunctionLetter^{\atEquilibrium}_{\pm}(\discrete{u}^{\vectorial{\varepsilon}, n}_j)
          +\ud\left(\exa-\exs\right)\left(
            \distributionFunctionLetter^{\atEquilibrium}_{\mp}(\discrete{u}^{\vectorial{\varepsilon}, n}_j)-\discrete{f}^{\vectorial{\varepsilon}, n}_{\mp,j}
          \right),
        \end{aligned}
      \right.
    \end{equation}
    for $j\in\relatives$.
    If we denote
    \begin{equation}\label{omegaId}
      \oa\definitionEquality1-\expo^{-\frac{\dt}{\ea}}
      \qquad\textnormal{and}\qquad
      \os\definitionEquality 1-\expo^{-\frac{\dt}{\es}},        
    \end{equation}
    we note that \eqref{collision} is nothing but the relaxation step for the TRT lattice Boltzmann scheme of \cite{aregba2026monotonicity}.
    However, here we want to look separately at the dependence on the relaxation parameters and on $\dt$.
    Note that $\oa, \os\in (0, 1)$.

    
    Let us now describe how \eqref{collision} is obtained: we solve exactly the ordinary differential system $\frac{\differential}{\differential\timeVariable}\vectorial{f}^{\vectorial{\varepsilon}}(\timeVariable)=\vectorial{G}(\vectorial{f}^{\vectorial{\varepsilon}}(\timeVariable))$ where $\vectorial{G}$ is the right-hand side in \eqref{system1}. 
  This is done by integrating the right-hand side of the equations on the moments \eqref{eq:systemMoments} and then go back to the distribution functions.
  From the first equation, we deduce that $u^{\vectorial{\varepsilon}}(t) = u^{\vectorial{\varepsilon}}(0)$, so that the second and third equations read
   \begin{equation}
      \left\{
  \begin{aligned}
  &\frac{\differential}{\differential\timeVariable}v^{\vectorial{\varepsilon}}(t) =-\frac{1}{\ea}\left(v^{\vectorial{\varepsilon}}(t)
    -\flux(u^{\vectorial{\varepsilon}}(t))\right) = (v^{\vectorial{\varepsilon}}(t)-\flux(u^{\vectorial{\varepsilon}}(0)))',\\
      &\frac{\differential}{\differential\timeVariable}w^{\vectorial{\varepsilon}}(t) =-\frac{1}{\es}\left(w^{\vectorial{\varepsilon}}(t) - 2\equilibriumCoefficientLinear u^{\vectorial{\varepsilon}}(t)\right)
      =(w^{\vectorial{\varepsilon}}(t)-2\equilibriumCoefficientLinear u^{\vectorial{\varepsilon}}(0))'.
  \end{aligned}
  \right.
\end{equation}
This system is integrated by 
\begin{equation*}
    v^{\vectorial{\varepsilon}}(t) = e^{-\frac{t}{\ea}} v^{\vectorial{\varepsilon}}(0) + (1-e^{-\frac{t}{\ea}})\flux(u^{\vectorial{\varepsilon}}(0))
    \qquad \textnormal{and}\qquad 
    w^{\vectorial{\varepsilon}}(t) = e^{-\frac{t}{\es}} w^{\vectorial{\varepsilon}}(0) + (1-e^{-\frac{t}{\es}})2\equilibriumCoefficientLinear  u^{\vectorial{\varepsilon}}(0).
\end{equation*}
Morally, we have been able to exactly integrate the system in a basis which is not the one of the distribution function, which would be the case in the BGK setting.
Coming back to the distribution functions and integrating over a time step $\timeStep$ gives \eqref{collision}.

    \item Transport: 
    \begin{equation}\label{tr}
      \discrete{f}^{\vectorial{\varepsilon}, n+1}_{i,j}=\discrete{f}^{\vectorial{\varepsilon}, n,\collided}_{i,j-c_i/\lambda}, 
      \qquad j\in\relatives, 
      \quad 
      i=\zeroVelocitySymbol, \positiveVelocitySymbol, \negativeVelocitySymbol.
    \end{equation}
\end{itemize}

\subsection{Monotonicity of the scheme and link with quasi-monotonicity of the relaxation system}
    The monotonicity (not only quasi-monotonicity) conditions of the scheme\footnote{More precisely, of the relaxation \eqref{collision}.} above are  as in \cite[Proposition 3.3]{aregba2026monotonicity}:
    \[
      \omega_s (1-2\equilibriumCoefficientLinear )\geq \max(0,\omega_s-1)
      \qquad \textnormal{and}\qquad
      \omega_a \max_{u\in[-\maximumInitialDatum, \maximumInitialDatum]}\frac{|\flux^\prime(u)|}{2\lambda}\leq \omega_s \equilibriumCoefficientLinear +\ud \min(2-\omega_a-\omega_s,0,\omega_a-\omega_s).
    \]
    This property is crucial to obtain the estimations on the numerical solution of \cite{aregba2026monotonicity}.
    Here the first condition becomes:
    \[
      (1-\exs)(1-2\equilibriumCoefficientLinear )\geq 0,
    \]
    since only the regime $\os< 1$ can be achieved under the identification by \eqref{omegaId}.
    This becomes $\equilibriumCoefficientLinear  \leq \ud$, and the second condition is
    \[
      (1-\exa)\max_{u\in[-\maximumInitialDatum, \maximumInitialDatum]}\frac{|\flux^\prime(u)|}{2\lambda}\leq(1-\exs)\equilibriumCoefficientLinear  +\ud \min(0,\exs-\exa).
    \]
    This entails
    \begin{equation*}
        \equilibriumCoefficientLinear  \geq 
\frac{1}{2(1-e^{-\frac{\timeStep}{\es}})}
\Bigl ( \max_{u\in[-\maximumInitialDatum, \maximumInitialDatum]}\frac{(1-e^{-\frac{\timeStep}{\ea}})|\flux^\prime(u)|}{\lambda} + (\exs-\exa)^- \Bigr )> 0,
    \end{equation*}
    therefore $\equilibriumCoefficientLinear  \in (0,\ud]$.

\begin{itemize}
    \item Let $\ea \leq \es$, hence
$\exs-\exa \geq 0$ and the monotonicity conditions that we look at are
\begin{equation}\label{L2}
  \left\{
    \begin{aligned}
      & \equilibriumCoefficientLinear  \in (0,\tfrac{1}{2}],\\
      &   (1-\exa)\max_{u\in[-\maximumInitialDatum, \maximumInitialDatum]}\frac{|\flux^\prime(u)|}{\lambda}\leq 2(1-\exs)\equilibriumCoefficientLinear .
    \end{aligned}
  \right.
\end{equation}
    \item Let $\ea \geq \es$, hence
$\exs-\exa \leq 0$ and the monotonicity
conditions become
\begin{equation}\label{L2Other}
  \left\{
    \begin{aligned}
      & \equilibriumCoefficientLinear  \in (0,\tfrac{1}{2}],\\
      &   (1-\exa)\max_{u\in[-\maximumInitialDatum, \maximumInitialDatum]}\frac{|\flux^\prime(u)|}{\lambda}\leq2(1-\exs)\equilibriumCoefficientLinear  
      +(1-e^{-\frac{\timeStep}{\ea}})
      -(1-e^{-\frac{\timeStep}{\es}}).
    \end{aligned}
  \right.
\end{equation}
\end{itemize}

\begin{proposition}[Quasi-monotonicity of the relaxation system implies monotonicity of the numerical scheme]
  If the assumptions of \Cref{mono1} are fulfilled, then \eqref{L2} (respectively \eqref{L2Other}) is satisfied.
  \end{proposition}
\begin{proof}
Let us start with the case $\ea \leq \es$.
The implication holds if we show that 
\begin{equation*}
    \frac{\ea}{\es}\leq \frac{1-e^{-\frac{\timeStep}{\es}}}{1-e^{-\frac{\timeStep}{\ea}}}.
\end{equation*}
This boils down to prove that for $y>0$, $g(y)=y(1-\expo^{-\frac{\tau}{y}})$ is an increasing function for all $\tau>0$.
We get
    \[
      g^\prime(y)=1-\expo^{-\frac{\tau}{y}}(1+\frac{\tau}{y}).
    \]
    with $g^\prime(0)=1$ and $g^\prime(+\infty)=0$. Moreover
    \[
      g^{\prime \prime}(y)=-\expo^{-\frac{\tau}{y}}\frac{\tau^2}{y^3}<0,
    \]
    which ends the first part of the proof.
    
    For the case $\ea \geq \es$, the proof would be achieved if we manage to show that 
    \begin{equation*}
        2\frac{\ea}{\es}\equilibriumCoefficientLinear  + (1-\frac{\ea}{\es})
        \leq 
        2\frac{1-e^{-\frac{\timeStep}{\es}}}{1-e^{-\frac{\timeStep}{\ea}}}\equilibriumCoefficientLinear 
        + \Bigl (1-\frac{1-e^{-\frac{\timeStep}{\es}}}{1-e^{-\frac{\timeStep}{\ea}}}\Bigr ).
    \end{equation*}
    Straightforward manipulations show that the previous inequality is equivalent to 
    \begin{equation*}
        (2\equilibriumCoefficientLinear  - 1)\ea (1-e^{-\frac{\timeStep}{\ea}})
    \leq 
    (2\equilibriumCoefficientLinear  - 1)\es( 1-e^{-\frac{\timeStep}{\es}}).
    \end{equation*}
    Since $2\equilibriumCoefficientLinear  - 1\leq 0$, this inequation holds true thanks to the increasing property of $g$.
\end{proof}

\subsection{Properties of the numerical solution}
 We introduce the following notations:
 \begin{equation*}
  \vectorial{f}_{\Delta}^{\vectorial{\varepsilon},\indexTime}(\spaceVariable)
  \definitionEquality
  \sum_{\indexSpace\in\relatives}
  \vectorial{\discrete{f}}_{\indexSpace}^{\vectorial{\varepsilon},\indexTime}\indicatorFunction{((\indexSpace-\frac{1}{2})\spaceStep,(\indexSpace+\frac{1}{2})\spaceStep) }(\spaceVariable)
  \qquad 
  \text{and}
  \qquad 
  \vectorial{f}^{\vectorial{\varepsilon}}_{\Delta}(\timeVariable, \spaceVariable)
  \definitionEquality
  \sum_{\indexTime\in\naturals}
  \vectorial{f}_{\Delta}^{\vectorial{\varepsilon},\indexTime}(\spaceVariable)
  \indicatorFunction{[\indexTime\timeStep, (\indexTime+1)\timeStep)}(\timeVariable)
 \end{equation*}
 and analogous ones for other quantities, such as the sum of the distribution functions.
 We stress that $(\ea, \es)$ are now fixed.
  Gathering Proposition 4.2, Proposition 4.5, and Corollary 4.6 from \cite{aregba2026monotonicity}, we obtain the following.
 \begin{proposition}\label{prop:propertiesEnsuringExtractionScheme}
  Assume that $u^{\initial}\in L^{\infty}(\reals)\cap L^{1}(\reals)\cap \textnormal{BV}(\reals)$.
  Let the assumptions of \Cref{mono1} be fulfilled.
  Denote $\vectorial{\discrete{g}}_{\indexSpace}^{\vectorial{\varepsilon},\indexTime}$ the numerical solution obtained in the same way as $\vectorial{\discrete{f}}_{\indexSpace}^{\vectorial{\varepsilon},\indexTime}$ from an initial datum $v^{\initial}$ such that $\lVert v^{\initial} \rVert_{\infty}\leq \maximumInitialDatum$.
  Then,
  \begin{align}
    &\vectorial{f}^{\vectorial{\varepsilon},n}_\Delta(\spaceVariable) \in [\minvec, \maxvec]
    \qquad \text{and}
    \qquad 
  u^{\vectorial{\varepsilon},n}_\Delta(\spaceVariable) \in [-\maximumInitialDatum, \maximumInitialDatum],\label{eq:compactInvariant}\\
  &\| \vectorial{f}^{\vectorial{\varepsilon},n+1}_\Delta - \vectorial{g}^{\vectorial{\varepsilon},n+1}_\Delta \|_1 = \| \vectorial{f}^{\vectorial{\varepsilon},n, \collided}_\Delta - \vectorial{g}^{\vectorial{\varepsilon},n, \collided}_\Delta \|_1 \leq \| \vectorial{f}^{\vectorial{\varepsilon},n}_\Delta - \vectorial{g}^{\vectorial{\varepsilon},n}_\Delta \|_1 \leq \|u^\initial - v^\initial\|_1, \qquad \text{and}\label{eqcontracL1}\\
  &\textnormal{TV}( \vectorial{f}^{\vectorial{\varepsilon},n+1}_\Delta )= \textnormal{TV}( \vectorial{f}^{\vectorial{\varepsilon},n, \collided}_\Delta )\leq \textnormal{TV}(\vectorial{f}^{\vectorial{\varepsilon},n}_\Delta ) \leq \textnormal{TV}(u^\initial).\label{eqTVD}
  \end{align}
  Moreover, there exists a constant $C_{\textnormal{ec}}>0$ independent of $\ea$ and $\es$ such that for all $t,t^\prime \geq 0$:
\begin{equation}\label{timeequic1}
    \| \vectorial{f}^{\vectorial{\varepsilon}}_\Delta(t, \cdot) - \vectorial{f}^{\vectorial{\varepsilon}}_\Delta(t^\prime, \cdot) \|_1 \leq C_{\textnormal{ec}}(|t-t^\prime|+\dt)\textnormal{TV}(u^\initial).
  \end{equation}
 \end{proposition}
With the exception of \eqref{eq:compactInvariant}, the proof of \Cref{prop:propertiesEnsuringExtractionScheme} revolves around the $\ell^1$-contractivity of the relaxation, which is sketchily discussed below for the sake of illustration.
\begin{proof}[Sketch of the proof of \Cref{prop:propertiesEnsuringExtractionScheme}]
  Consider $\vectorial{\discrete{f}}, \vectorial{\discrete{g}} \in [\minvec, \maxvec]$ and set $\discrete{u}=\sum_{\indexVelocity}\discrete{f}_{\indexVelocity}\in [-\maximumInitialDatum, \maximumInitialDatum]$ and $\discrete{v}=\sum_{\indexVelocity}\discrete{g}_{\indexVelocity}\in [-\maximumInitialDatum, \maximumInitialDatum]$ (we drop the superscripts $\indexTime$ and $\vectorial{\varepsilon}$, and the subscript $\indexSpace$).
  We obtain, using the relaxation operator  
  \begin{align*}
    \discrete{f}_{\zeroVelocitySymbol}^{\collided}
    -
    \discrete{g}_{\zeroVelocitySymbol}^{\collided}
    &=
    (1-2\relaxationParameterSymmetric \equilibriumCoefficientLinear)
    (
      \discrete{f}_{\zeroVelocitySymbol}
    -
    \discrete{g}_{\zeroVelocitySymbol}
    )
    +\relaxationParameterSymmetric
    (1-2\equilibriumCoefficientLinear)
    (
      (
      \discrete{f}_{\positiveVelocitySymbol}
    -
    \discrete{g}_{\positiveVelocitySymbol}
    )
    +
    (
      \discrete{f}_{\negativeVelocitySymbol}
    -
    \discrete{g}_{\negativeVelocitySymbol}
    )
    ), \\
  \discrete{f}_{\positiveVelocitySymbol}^{\collided}
    -
    \discrete{g}_{\positiveVelocitySymbol}^{\collided}
    &=
    \bigl ( 1-\tfrac{1}{2}(\relaxationParameterSymmetric+\relaxationParameterAntiSymmetric)\bigr )
    (
      \discrete{f}_{\positiveVelocitySymbol}
    -
    \discrete{g}_{\positiveVelocitySymbol}
    )
    +\tfrac{1}{2}(\relaxationParameterAntiSymmetric-\relaxationParameterSymmetric)
    (
      \discrete{f}_{\negativeVelocitySymbol}
    -
    \discrete{g}_{\negativeVelocitySymbol}
    )
    +\relaxationParameterSymmetric\equilibriumCoefficientLinear (\discrete{u}-\discrete{v})
    +\frac{\relaxationParameterAntiSymmetric}{2\latticeVelocity}(\flux(\discrete{u})-\flux(\discrete{v})), \\
    \discrete{f}_{\negativeVelocitySymbol}^{\collided}
    -
    \discrete{g}_{\negativeVelocitySymbol}^{\collided}
    &=
    \bigl ( 1-\tfrac{1}{2}(\relaxationParameterSymmetric+\relaxationParameterAntiSymmetric)\bigr )
    (
      \discrete{f}_{\negativeVelocitySymbol}
    -
    \discrete{g}_{\negativeVelocitySymbol}
    )
    +\tfrac{1}{2}(\relaxationParameterAntiSymmetric-\relaxationParameterSymmetric)
    (
      \discrete{f}_{\positiveVelocitySymbol}
    -
    \discrete{g}_{\positiveVelocitySymbol}
    )
    +\relaxationParameterSymmetric\equilibriumCoefficientLinear (\discrete{u}-\discrete{v})
    -\frac{\relaxationParameterAntiSymmetric}{2\latticeVelocity}(\flux(\discrete{u})-\flux(\discrete{v})).
  \end{align*}
  By virtue of the mean value theorem, denoting $\discrete{w}_{\vartheta} = \vartheta\discrete{u}+(1-\vartheta)\discrete{v}\in [-\maximumInitialDatum, \maximumInitialDatum]$ for $\vartheta\in [0, 1]$, we have 
    \begin{align*}
  \discrete{f}_{\pm}^{\collided}
    -
    \discrete{g}_{\pm}^{\collided}
    =
    (
      \discrete{f}_{\zeroVelocitySymbol}
    -
    \discrete{g}_{\zeroVelocitySymbol}
    )
    \int_0^1\Bigl ( \relaxationParameterSymmetric\equilibriumCoefficientLinear
    \pm
    \frac{\relaxationParameterAntiSymmetric}{2\latticeVelocity}
    \flux'(\discrete{w}_{\vartheta})
    \Bigr )\differential{\vartheta}
    +
    &(
      \discrete{f}_{\pm}
    -
    \discrete{g}_{\pm}
    )
    \int_0^1\Bigl ( 1-\tfrac{1}{2}(\relaxationParameterSymmetric+\relaxationParameterAntiSymmetric )+ \relaxationParameterSymmetric\equilibriumCoefficientLinear\pm
    \frac{\relaxationParameterAntiSymmetric}{2\latticeVelocity}
    \flux'(\discrete{w}_{\vartheta})\Bigr )\differential{\vartheta}
    \\
    +
    &(
      \discrete{f}_{\mp}
    -
    \discrete{g}_{\mp}
    )
    \int_0^1\Bigl (\tfrac{1}{2}(\relaxationParameterAntiSymmetric-\relaxationParameterSymmetric)
    +\relaxationParameterSymmetric\equilibriumCoefficientLinear
    \pm
    \frac{\relaxationParameterAntiSymmetric}{2\latticeVelocity}
    \flux'(\discrete{w}_{\vartheta})
    \Bigr )
    \differential{\vartheta}.
    \end{align*}
    By monotonicity, the integrands above are non-negative, hence by triangle inequality
        \begin{align*}
    |\discrete{f}_{\pm}^{\collided}
    -
    \discrete{g}_{\pm}^{\collided}|
    \leq
    |
      \discrete{f}_{\zeroVelocitySymbol}
    -
    \discrete{g}_{\zeroVelocitySymbol}
    |
    \int_0^1\Bigl ( \relaxationParameterSymmetric\equilibriumCoefficientLinear
    \pm
    \frac{\relaxationParameterAntiSymmetric}{2\latticeVelocity}
    \flux'(\discrete{w}_{\vartheta})
    \Bigr )\differential{\vartheta}
    +
    &|
      \discrete{f}_{\pm}
    -
    \discrete{g}_{\pm}
    |
    \int_0^1\Bigl ( 1-\tfrac{1}{2}(\relaxationParameterSymmetric+\relaxationParameterAntiSymmetric )+ \relaxationParameterSymmetric\equilibriumCoefficientLinear\pm
    \frac{\relaxationParameterAntiSymmetric}{2\latticeVelocity}
    \flux'(\discrete{w}_{\vartheta})\Bigr )\differential{\vartheta}
    \\
    +
    &|
      \discrete{f}_{\mp}
    -
    \discrete{g}_{\mp}
    |
    \int_0^1\Bigl (\tfrac{1}{2}(\relaxationParameterAntiSymmetric-\relaxationParameterSymmetric)
    +\relaxationParameterSymmetric\equilibriumCoefficientLinear
    \pm
    \frac{\relaxationParameterAntiSymmetric}{2\latticeVelocity}
    \flux'(\discrete{w}_{\vartheta})
    \Bigr )
    \differential{\vartheta}.
    \end{align*}
    Summing, simple algebra delivers
    \begin{equation*}
      \sum_{\indexVelocity}
      |\discrete{f}_{\indexVelocity}^{\collided}
    -
    \discrete{g}_{\indexVelocity}^{\collided}|=
      |\discrete{f}_{\zeroVelocitySymbol}^{\collided}
    -
    \discrete{g}_{\zeroVelocitySymbol}^{\collided}|
    +
    |\discrete{f}_{\positiveVelocitySymbol}^{\collided}
    -
    \discrete{g}_{\positiveVelocitySymbol}^{\collided}|
    +
    |\discrete{f}_{\negativeVelocitySymbol}^{\collided}
    -
    \discrete{g}_{\negativeVelocitySymbol}^{\collided}|
    \leq
     |\discrete{f}_{\zeroVelocitySymbol}
    -
    \discrete{g}_{\zeroVelocitySymbol}|
    +
    |\discrete{f}_{\positiveVelocitySymbol}
    -
    \discrete{g}_{\positiveVelocitySymbol}|
    +
    |\discrete{f}_{\negativeVelocitySymbol}
    -
    \discrete{g}_{\negativeVelocitySymbol}|
    =
    \sum_{\indexVelocity}
      |\discrete{f}_{\indexVelocity}
    -
    \discrete{g}_{\indexVelocity}|,
    \end{equation*}
    which is the $\ell^1$-contractivity of the relaxation.
    From this, \eqref{eqcontracL1}, \eqref{eqTVD}, and \eqref{timeequic1} follow taking into account that the transport phase shifts the distribution function without mixing them.
\end{proof}
 Moreover, again from \cite{aregba2026monotonicity} (see Section 4.2), we inherit the following discrete entropy inequality.
 \begin{proposition}\label{prop:discreteEntropy}
  Assume that $u^{\initial}\in L^{\infty}(\reals)\cap L^{1}(\reals)\cap \textnormal{BV}(\reals)$.
  Let the assumptions of \Cref{mono1} be fulfilled.
  Then, for all $\kappa\in\reals$ and for every $\psi\in C_{\textnormal{c}}^{\infty}((0, +\infty)\times \reals)$ such that $\psi\geq 0$, we have 
  \begin{multline*}
    \int_{\timeStep}^{+\infty}
    \int_{-\infty}^{+\infty}
    \sum_{\indexVelocity}
    |\distributionFunctionLetter_{\Delta, \indexVelocity}^{\vectorial{\varepsilon}, \collided}(\timeVariable, \spaceVariable)-\distributionFunctionLetter_{\indexVelocity}^{\atEquilibrium}(\kappa)|
    \frac{\psi_{\Delta}(\timeVariable-\timeStep, \spaceVariable)-\psi_{\Delta}(\timeVariable, \spaceVariable)}{\timeStep}
    \differential\spaceVariable\differential\timeVariable
    \\
    \leq
    \int_{\timeStep}^{+\infty}
    \int_{-\infty}^{+\infty}
    \sum_{\indexVelocity}
    |\distributionFunctionLetter_{\Delta, \indexVelocity}^{\vectorial{\varepsilon}, \collided}(\timeVariable, \spaceVariable)-\distributionFunctionLetter_{\indexVelocity}^{\atEquilibrium}(\kappa)|
    \frac{\psi_{\Delta}(\timeVariable, \spaceVariable + \discreteVelocityLetter_{\indexVelocity} \spaceStep/\latticeVelocity)-\psi_{\Delta}(\timeVariable, \spaceVariable)}{\spaceStep/\latticeVelocity}
    \differential\spaceVariable\differential\timeVariable.
  \end{multline*}
  
 \end{proposition}
 This ends the results that we can extract using the discussion on two-relaxation-times lattice Boltzmann schemes that is presented in  \cite{aregba2026monotonicity}.
 In particular, these results are useless to deduce closeness of the numerical solution to equilibrium that could be used when $\timeStep\to 0$, due to the identification \eqref{omegaId}.
 We now rather have to follow the lines of \cite{MR1766856}, giving analogous results to their Lemma 5.4 and Proposition 5.5.
 \begin{lemma}
Assume that $u^{\initial}\in L^{\infty}(\reals)\cap L^{1}(\reals)\cap \textnormal{BV}(\reals)$.
  Let the assumptions of \Cref{mono1} be fulfilled.  Then, for $\indexTime\geq 1$
  \begin{equation}\label{eq:closenessEquilibrium}
    \lVert
    \vectorial{\distributionFunctionLetter}^{\atEquilibrium}
    (u_{\Delta}^{\vectorial{\varepsilon}, \indexTime, \collided})
    -
    \vectorial{\distributionFunctionLetter}_{\Delta}^{\vectorial{\varepsilon}, \indexTime, \collided}
    \rVert_1
    \leq 
    e^{-\frac{\timeStep}{\max(\ea, \es)}}
    \Bigl ( 
    \lVert
    \vectorial{\distributionFunctionLetter}^{\atEquilibrium}
    (u_{\Delta}^{\vectorial{\varepsilon}, \indexTime-1, \collided})
    -
    \vectorial{\distributionFunctionLetter}_{\Delta}^{\vectorial{\varepsilon}, \indexTime-1, \collided}
    \rVert_1
    +
    2 \latticeVelocity\timeStep \, \textnormal{TV}(u^\initial)
    \Bigr ) .
  \end{equation}
 \end{lemma}
 \begin{proof}
  Using the local conservation of $u$ during the relaxation, we obtain
  \begin{multline*}
    \lVert
    \vectorial{\distributionFunctionLetter}^{\atEquilibrium}
    (u_{\Delta}^{\vectorial{\varepsilon}, \indexTime, \collided})
    -
    \vectorial{\distributionFunctionLetter}_{\Delta}^{\vectorial{\varepsilon}, \indexTime, \collided}
    \rVert_1
    =
    \spaceStep\sum_{\indexSpace\in\relatives}
    \sum_{\indexVelocity}
    |
    \distributionFunctionLetter_{\indexVelocity}^{\atEquilibrium}(\discrete{u}_{\indexSpace}^{\vectorial{\varepsilon}, \indexTime})
    -\distributionFunction_{\indexVelocity, \indexSpace}^{\vectorial{\varepsilon}, \indexTime, \collided}
    |\\
    \leq 
    \spaceStep\sum_{\indexSpace\in\relatives}
    \Biggl ( 
    e^{-\frac{\timeStep}{\es}}
    |
    \distributionFunctionLetter_{\zeroVelocitySymbol}^{\atEquilibrium}(\discrete{u}_{\indexSpace}^{\vectorial{\varepsilon}, \indexTime})
    -\distributionFunction_{\zeroVelocitySymbol, \indexSpace}^{\vectorial{\varepsilon}, \indexTime}|
    +\tfrac{1}{2}
    \Bigl ( 
      e^{-\frac{\timeStep}{\ea}}+e^{-\frac{\timeStep}{\es}}
      +
      |e^{-\frac{\timeStep}{\ea}}-e^{-\frac{\timeStep}{\es}}|
      \Bigr )
      \Bigl ( 
|
    \distributionFunctionLetter_{\positiveVelocitySymbol}^{\atEquilibrium}(\discrete{u}_{\indexSpace}^{\vectorial{\varepsilon}, \indexTime})
    -\distributionFunction_{\positiveVelocitySymbol, \indexSpace}^{\vectorial{\varepsilon}, \indexTime}|
    +
    |
    \distributionFunctionLetter_{\negativeVelocitySymbol}^{\atEquilibrium}(\discrete{u}_{\indexSpace}^{\vectorial{\varepsilon}, \indexTime})
    -\distributionFunction_{\negativeVelocitySymbol, \indexSpace}^{\vectorial{\varepsilon}, \indexTime}|
      \Bigr )
    \Biggr ),
  \end{multline*}
  where the inequality comes the triangle inequality after using \eqref{collision}.
  Distinguishing the case $\ea\leq \es$ from $\ea\geq \es$, one deduces that 
  \begin{equation*}
    \lVert
    \vectorial{\distributionFunctionLetter}^{\atEquilibrium}
    (u_{\Delta}^{\vectorial{\varepsilon}, \indexTime, \collided})
    -
    \vectorial{\distributionFunctionLetter}_{\Delta}^{\vectorial{\varepsilon}, \indexTime, \collided}
    \rVert_1
    \leq 
    \spaceStep 
    \max(e^{-\frac{\timeStep}{\ea}}, e^{-\frac{\timeStep}{\es}})
    \sum_{\indexSpace\in\relatives}
    \sum_{\indexVelocity}
    |
    \distributionFunctionLetter_{\indexVelocity}^{\atEquilibrium}(\discrete{u}_{\indexSpace}^{\vectorial{\varepsilon}, \indexTime})
    -\distributionFunction_{\indexVelocity, \indexSpace}^{\vectorial{\varepsilon}, \indexTime}|
    =
    \spaceStep 
    e^{-\frac{\timeStep}{\max(\ea, \es)}}
    \sum_{\indexSpace\in\relatives}
    \sum_{\indexVelocity}
    |
    \distributionFunctionLetter_{\indexVelocity}^{\atEquilibrium}(\discrete{u}_{\indexSpace}^{\vectorial{\varepsilon}, \indexTime})
    -\distributionFunction_{\indexVelocity, \indexSpace}^{\vectorial{\varepsilon}, \indexTime}|.
  \end{equation*}
  We now take transport into account. 
  Let us point out that $\conservedVariableDiscrete_{\indexSpace}^{\vectorial{\varepsilon}, \indexTime} = \distributionFunction_{\zeroVelocitySymbol, \indexSpace}^{\vectorial{\varepsilon}, \indexTime-1, \collided} + \distributionFunction_{\positiveVelocitySymbol, \indexSpace-1}^{\vectorial{\varepsilon}, \indexTime-1, \collided} + \distributionFunction_{\negativeVelocitySymbol, \indexSpace+1}^{\vectorial{\varepsilon}, \indexTime-1, \collided}$, which yields 
  \begin{align*}
    \distributionFunctionLetter_{\zeroVelocitySymbol}^{\atEquilibrium}(\conservedVariableDiscrete_{\indexSpace}^{\vectorial{\varepsilon}, \indexTime})
    =
    \distributionFunctionLetter_{\zeroVelocitySymbol}^{\atEquilibrium}(\conservedVariableDiscrete_{\indexSpace}^{\vectorial{\varepsilon}, \indexTime-1, \collided}
    -\distributionFunction_{\positiveVelocitySymbol, \indexSpace}^{\vectorial{\varepsilon}, \indexTime-1, \collided}+\distributionFunction_{\positiveVelocitySymbol, \indexSpace-1}^{\vectorial{\varepsilon}, \indexTime-1, \collided}
    -\distributionFunction_{\negativeVelocitySymbol, \indexSpace}^{\vectorial{\varepsilon}, \indexTime-1, \collided}+\distributionFunction_{\negativeVelocitySymbol, \indexSpace+1}^{\vectorial{\varepsilon}, \indexTime-1, \collided}),\\
    \distributionFunctionLetter_{\positiveVelocitySymbol}^{\atEquilibrium}(\conservedVariableDiscrete_{\indexSpace}^{\vectorial{\varepsilon}, \indexTime})
    =
    \distributionFunctionLetter_{\positiveVelocitySymbol}^{\atEquilibrium}(\conservedVariableDiscrete_{\indexSpace-1}^{\vectorial{\varepsilon}, \indexTime-1, \collided}
    -\distributionFunction_{\zeroVelocitySymbol, \indexSpace-1}^{\vectorial{\varepsilon}, \indexTime-1, \collided} + \distributionFunction_{\zeroVelocitySymbol, \indexSpace}^{\vectorial{\varepsilon}, \indexTime-1, \collided}
    -\distributionFunction_{\negativeVelocitySymbol, \indexSpace-1}^{\vectorial{\varepsilon}, \indexTime-1, \collided}+\distributionFunction_{\negativeVelocitySymbol, \indexSpace+1}^{\vectorial{\varepsilon}, \indexTime-1, \collided}),
    \\
    \distributionFunctionLetter_{\negativeVelocitySymbol}^{\atEquilibrium}(\conservedVariableDiscrete_{\indexSpace}^{\vectorial{\varepsilon}, \indexTime})
    =
    \distributionFunctionLetter_{\negativeVelocitySymbol}^{\atEquilibrium}(\conservedVariableDiscrete_{\indexSpace+1}^{\vectorial{\varepsilon}, \indexTime-1, \collided}
    -\distributionFunction_{\zeroVelocitySymbol, \indexSpace+1}^{\vectorial{\varepsilon}, \indexTime-1, \collided} + \distributionFunction_{\zeroVelocitySymbol, \indexSpace}^{\vectorial{\varepsilon}, \indexTime-1, \collided}
    -\distributionFunction_{\positiveVelocitySymbol, \indexSpace+1}^{\vectorial{\varepsilon}, \indexTime-1, \collided}+\distributionFunction_{\positiveVelocitySymbol, \indexSpace-1}^{\vectorial{\varepsilon}, \indexTime-1, \collided}).
  \end{align*}
  The fact that $\conservedVariableDiscrete_{\indexSpace}^{\vectorial{\varepsilon}, \indexTime}, \conservedVariableDiscrete_{\indexSpace}^{\vectorial{\varepsilon}, \indexTime-1, \collided}, \conservedVariableDiscrete_{\indexSpace\mp 1}^{\vectorial{\varepsilon}, \indexTime-1, \collided} \in [-\maximumInitialDatum, \maximumInitialDatum]$ ensures that---by virtue of the mean value theorem---there exist $\tilde{\conservedVariable}_1, \tilde{\conservedVariable}_2, \tilde{\conservedVariable}_3 \in (-\maximumInitialDatum, \maximumInitialDatum)$ such that 
    \begin{align*}
    \distributionFunctionLetter_{\zeroVelocitySymbol}^{\atEquilibrium}(\conservedVariableDiscrete_{\indexSpace}^{\vectorial{\varepsilon}, \indexTime})
    =
    \distributionFunctionLetter_{\zeroVelocitySymbol}^{\atEquilibrium}(\conservedVariableDiscrete_{\indexSpace}^{\vectorial{\varepsilon}, \indexTime-1, \collided})
    +
    \frac{\differential\distributionFunctionLetter_{\zeroVelocitySymbol}^{\atEquilibrium}(\tilde{\conservedVariable}_1)}{\differential\conservedVariable}
    (-\distributionFunction_{\positiveVelocitySymbol, \indexSpace}^{\vectorial{\varepsilon}, \indexTime-1, \collided}+\distributionFunction_{\positiveVelocitySymbol, \indexSpace-1}^{\vectorial{\varepsilon}, \indexTime-1, \collided}
    -\distributionFunction_{\negativeVelocitySymbol, \indexSpace}^{\vectorial{\varepsilon}, \indexTime-1, \collided}+\distributionFunction_{\negativeVelocitySymbol, \indexSpace+1}^{\vectorial{\varepsilon}, \indexTime-1, \collided}),\\
    \distributionFunctionLetter_{\positiveVelocitySymbol}^{\atEquilibrium}(\conservedVariableDiscrete_{\indexSpace}^{\vectorial{\varepsilon}, \indexTime})
    =
    \distributionFunctionLetter_{\positiveVelocitySymbol}^{\atEquilibrium}(\conservedVariableDiscrete_{\indexSpace-1}^{\vectorial{\varepsilon}, \indexTime-1, \collided})
    +
    \frac{\differential\distributionFunctionLetter_{\positiveVelocitySymbol}^{\atEquilibrium}(\tilde{\conservedVariable}_2)}{\differential\conservedVariable}(
    -\distributionFunction_{\zeroVelocitySymbol, \indexSpace-1}^{\vectorial{\varepsilon}, \indexTime-1, \collided} + \distributionFunction_{\zeroVelocitySymbol, \indexSpace}^{\vectorial{\varepsilon}, \indexTime-1, \collided}
    -\distributionFunction_{\negativeVelocitySymbol, \indexSpace-1}^{\vectorial{\varepsilon}, \indexTime-1, \collided}+\distributionFunction_{\negativeVelocitySymbol, \indexSpace+1}^{\vectorial{\varepsilon}, \indexTime-1, \collided}),
    \\
    \distributionFunctionLetter_{\negativeVelocitySymbol}^{\atEquilibrium}(\conservedVariableDiscrete_{\indexSpace}^{\vectorial{\varepsilon}, \indexTime})
    =
    \distributionFunctionLetter_{\negativeVelocitySymbol}^{\atEquilibrium}(\conservedVariableDiscrete_{\indexSpace+1}^{\vectorial{\varepsilon}, \indexTime-1, \collided})
    +
    \frac{\differential\distributionFunctionLetter_{\negativeVelocitySymbol}^{\atEquilibrium}(\tilde{\conservedVariable}_3)}{\differential\conservedVariable}
    (-\distributionFunction_{\zeroVelocitySymbol, \indexSpace+1}^{\vectorial{\varepsilon}, \indexTime-1, \collided} + \distributionFunction_{\zeroVelocitySymbol, \indexSpace}^{\vectorial{\varepsilon}, \indexTime-1, \collided}
    -\distributionFunction_{\positiveVelocitySymbol, \indexSpace+1}^{\vectorial{\varepsilon}, \indexTime-1, \collided}+\distributionFunction_{\positiveVelocitySymbol, \indexSpace-1}^{\vectorial{\varepsilon}, \indexTime-1, \collided}).
  \end{align*}
  We recall that 
  \begin{equation*}
    0\leq \frac{\differential\distributionFunctionLetter_{\indexVelocity}^{\atEquilibrium}(\tilde{\conservedVariable}_{\indexVelocity})}{\differential\conservedVariable}
    \leq 1, 
    \qquad \indexVelocity=\zeroVelocitySymbol, \positiveVelocitySymbol, \negativeVelocitySymbol.
  \end{equation*}
  Overall, this discussion provides 
  \begin{multline*}
    \sum_{\indexVelocity}
    |
    \distributionFunctionLetter_{\indexVelocity}^{\atEquilibrium}(\discrete{u}_{\indexSpace}^{\vectorial{\varepsilon}, \indexTime})
    -\distributionFunction_{\indexVelocity, \indexSpace}^{\vectorial{\varepsilon}, \indexTime}|
    \leq 
    \sum_{\indexVelocity}
    |
    \distributionFunctionLetter_{\indexVelocity}^{\atEquilibrium}(\discrete{u}_{\indexSpace-\discreteVelocityLetter_{\indexVelocity}/\latticeVelocity}^{\vectorial{\varepsilon}, \indexTime-1, \collided})
    -\distributionFunction_{\indexVelocity, \indexSpace-\discreteVelocityLetter_{\indexVelocity}/\latticeVelocity}^{\vectorial{\varepsilon}, \indexTime-1, \collided}|
    +
    |\distributionFunction_{\positiveVelocitySymbol, \indexSpace-1}^{\vectorial{\varepsilon}, \indexTime-1, \collided}-\distributionFunction_{\positiveVelocitySymbol, \indexSpace}^{\vectorial{\varepsilon}, \indexTime-1, \collided}|
    +|\distributionFunction_{\negativeVelocitySymbol, \indexSpace+1}^{\vectorial{\varepsilon}, \indexTime-1, \collided}-\distributionFunction_{\negativeVelocitySymbol, \indexSpace}^{\vectorial{\varepsilon}, \indexTime-1, \collided} |\\
    +|
    \distributionFunction_{\zeroVelocitySymbol, \indexSpace}^{\vectorial{\varepsilon}, \indexTime-1, \collided}
    -\distributionFunction_{\zeroVelocitySymbol, \indexSpace-1}^{\vectorial{\varepsilon}, \indexTime-1, \collided} |+|
    \distributionFunction_{\negativeVelocitySymbol, \indexSpace+1}^{\vectorial{\varepsilon}, \indexTime-1, \collided}-\distributionFunction_{\negativeVelocitySymbol, \indexSpace-1}^{\vectorial{\varepsilon}, \indexTime-1, \collided}|
    +
    |\distributionFunction_{\zeroVelocitySymbol, \indexSpace}^{\vectorial{\varepsilon}, \indexTime-1, \collided}-\distributionFunction_{\zeroVelocitySymbol, \indexSpace+1}^{\vectorial{\varepsilon}, \indexTime-1, \collided}|
    +
    |\distributionFunction_{\positiveVelocitySymbol, \indexSpace-1}^{\vectorial{\varepsilon}, \indexTime-1, \collided}-\distributionFunction_{\positiveVelocitySymbol, \indexSpace+1}^{\vectorial{\varepsilon}, \indexTime-1, \collided}|.
  \end{multline*}
  By changes in spatial indices and adding non-negative quantities on the right-hand side, we have 
  \begin{align}
    \sum_{\indexSpace\in\relatives}
    \sum_{\indexVelocity}
    |
    \distributionFunctionLetter_{\indexVelocity}^{\atEquilibrium}(\discrete{u}_{\indexSpace}^{\vectorial{\varepsilon}, \indexTime})
    -\distributionFunction_{\indexVelocity, \indexSpace}^{\vectorial{\varepsilon}, \indexTime}|
    &\leq 
    \sum_{\indexSpace\in\relatives}
    \sum_{\indexVelocity}
    |
    \distributionFunctionLetter_{\indexVelocity}^{\atEquilibrium}(\discrete{u}_{\indexSpace}^{\vectorial{\varepsilon}, \indexTime-1, \collided})
    -\distributionFunction_{\indexVelocity, \indexSpace}^{\vectorial{\varepsilon}, \indexTime-1, \collided}|
    +
    2 \, \textnormal{TV}( \vectorial{f}^{\vectorial{\varepsilon},n-1, \collided}_\Delta )\nonumber
    \\
    &\leq 
    \sum_{\indexSpace\in\relatives}
    \sum_{\indexVelocity}
    |
    \distributionFunctionLetter_{\indexVelocity}^{\atEquilibrium}(\discrete{u}_{\indexSpace}^{\vectorial{\varepsilon}, \indexTime-1, \collided})
    -\distributionFunction_{\indexVelocity, \indexSpace}^{\vectorial{\varepsilon}, \indexTime-1, \collided}|
    +
    2 \, \textnormal{TV}(u^\initial),\label{eq:tmp2}
  \end{align}
  with the second inequality coming from \eqref{eqTVD}.
  This gives the claim.
 \end{proof}

\begin{proposition}\label{prop:deltaFromEquil}
Assume that $u^{\initial}\in L^{\infty}(\reals)\cap L^{1}(\reals)\cap \textnormal{BV}(\reals)$.
  Let the assumptions of \Cref{mono1} be fulfilled.
  For all $\timeVariable\geq 0$,
  \begin{equation*}
    \lVert
    \vectorial{\distributionFunctionLetter}^{\atEquilibrium}(\conservedVariable_{\Delta}^{\vectorial{\varepsilon}}(\timeVariable, \cdot))
    -
    \vectorial{\distributionFunctionLetter}_{\Delta}^{\vectorial{\varepsilon}}(\timeVariable, \cdot)
    \rVert_{1}
    \leq
    2 \latticeVelocity  \textnormal{TV}(u^\initial)
    (\max(\ea, \es)+\timeStep).
  \end{equation*}
\end{proposition}
\begin{proof}
  Iterating from \eqref{eq:closenessEquilibrium} and using the fact that $\lVert
    \vectorial{\distributionFunctionLetter}^{\atEquilibrium}
    (u_{\Delta}^{\vectorial{\varepsilon}, 0, \collided})
    -
    \vectorial{\distributionFunctionLetter}_{\Delta}^{\vectorial{\varepsilon}, 0, \collided}
    \rVert_1 = 0$, we obtain 
  \begin{multline*}
    \lVert
    \vectorial{\distributionFunctionLetter}^{\atEquilibrium}
    (u_{\Delta}^{\vectorial{\varepsilon}, \indexTime, \collided})
    -
    \vectorial{\distributionFunctionLetter}_{\Delta}^{\vectorial{\varepsilon}, \indexTime, \collided}
    \rVert_1
    \leq 
    2 \latticeVelocity  \textnormal{TV}(u^\initial)
    \timeStep
    \,
    e^{-\frac{\timeStep}{\max(\ea, \es)}}
    \sum_{k=0}^{\indexTime-1}
    \Bigl (e^{-\frac{\timeStep}{\max(\ea, \es)}}
    \Bigr )^k
    =
    2 \latticeVelocity  \textnormal{TV}(u^\initial)
    \timeStep
    \,
    e^{-\frac{\timeStep}{\max(\ea, \es)}}
    \frac{1-e^{-\indexTime\frac{\timeStep}{\max(\ea, \es)}}}{1-e^{-\frac{\timeStep}{\max(\ea, \es)}}}\\
    \leq 
    2 \latticeVelocity  \textnormal{TV}(u^\initial)
    \max(\ea, \es) \,
    \frac{\timeStep}{\max(\ea, \es)}    
    \frac{e^{-\frac{\timeStep}{\max(\ea, \es)}}}{1-e^{-\frac{\timeStep}{\max(\ea, \es)}}}
    =
    2 \latticeVelocity  \textnormal{TV}(u^\initial)
    \max(\ea, \es) 
    \, q\Bigl (\frac{\timeStep}{\max(\ea, \es)}\Bigr ),
  \end{multline*}
  where $q(y) = y \, e^{-y}/(1-e^{-y})$.
  One can easily show that $0< q(y) \leq 1$ for $y\geq 0$, hence $\lVert
    \vectorial{\distributionFunctionLetter}^{\atEquilibrium}
    (u_{\Delta}^{\vectorial{\varepsilon}, \indexTime, \collided})
    -
    \vectorial{\distributionFunctionLetter}_{\Delta}^{\vectorial{\varepsilon}, \indexTime, \collided}
    \rVert_1 \leq 2 \latticeVelocity  \textnormal{TV}(u^\initial)
    \max(\ea, \es) $.
    Using \eqref{eq:tmp2} yields the final claim.
\end{proof}

\subsection{Convergence of the numerical scheme and estimates on its limit}

We are now in position to give the following result, which is partially analogous to \cite[Theorem 5.6]{MR1766856}.
\begin{theorem}\label{thm:convergenceNumericsPlusEstimates}
  Let $\ea, \es>0$ be given.
  Assume that $u^{\initial}\in L^{\infty}(\reals)\cap L^{1}(\reals)\cap \textnormal{BV}(\reals)$.
  Let the assumptions of \Cref{mono1} be fulfilled.
For any $T>0$, the numerical solution $\vectorial{f}_{\Delta}^{\vectorial{\varepsilon}}$ converges in $L^{\infty}([0, T]; L^1_{\textnormal{loc}}(\reals)^3)$ to $\vectorial{\distributionFunctionLetter}^{\vectorial{\varepsilon}}\in C^0([0, T]; L^1(\reals)^3)$, the unique solution of the Cauchy problem \eqref{system1}--\eqref{eq:equilibriaForm}--\eqref{f0} and satisfies the following estimates:
      \begin{align}
    \vectorial{f}^{\vectorial{\varepsilon}}(\timeVariable, \spaceVariable) &\in [\minvec, \maxvec], \qquad 
    \textnormal{for almost every }(\timeVariable, \spaceVariable)\in [0, T]\times \reals,\\
  \| \vectorial{f}^{\vectorial{\varepsilon}}(\timeVariable, \cdot) \|_1 &\leq \|u^\initial\|_1, 
  \qquad \text{for all }\timeVariable\in[0, T],\\
  \textnormal{TV}( \vectorial{f}^{\vectorial{\varepsilon}}(\timeVariable, \cdot) ) &\leq \textnormal{TV}(u^\initial), \qquad \text{for all }t\in[0, T], \\
  \| \vectorial{f}^{\vectorial{\varepsilon}}(t, \cdot) - \vectorial{f}^{\vectorial{\varepsilon}}(t^\prime, \cdot) \|_1 &\leq C_{\textnormal{ec}}|t-t^\prime|\textnormal{TV}(u^\initial), \qquad \text{for all }t, t'\in [0, T], \\
  \lVert
    \vectorial{\distributionFunctionLetter}^{\atEquilibrium}(
      {\distributionFunctionLetter}_1^{\vectorial{\varepsilon}}(\timeVariable, \cdot)
      +{\distributionFunctionLetter}_2^{\vectorial{\varepsilon}}(\timeVariable, \cdot)
      +{\distributionFunctionLetter}_3^{\vectorial{\varepsilon}}(\timeVariable, \cdot)
    )
    -
    \vectorial{\distributionFunctionLetter}^{\vectorial{\varepsilon}}(\timeVariable, \cdot)
    \rVert_{1}
    &\leq
    2 \latticeVelocity  \textnormal{TV}(u^\initial)
    \max(\ea, \es), \qquad \text{for all }t\in[0, T].\label{eq:almostEquilibrium}
  \end{align}
  Finally, for all $\kappa\in\reals$ and for every $\psi\in C_{\textnormal{c}}^{\infty}((0, T)\times \reals)$ such that $\psi\geq 0$, we have 
  \begin{equation}\label{eq:krushkovFixedEps}
    -\int_{0}^{T}
    \int_{-\infty}^{+\infty}
    \sum_{\indexVelocity}
    |\distributionFunctionLetter_{\indexVelocity}^{\vectorial{\varepsilon}}(\timeVariable, \spaceVariable)-\distributionFunctionLetter_{\indexVelocity}^{\atEquilibrium}(\kappa)|
    \partial_{\timeVariable}\psi(\timeVariable, \spaceVariable)
    \differential\spaceVariable\differential\timeVariable
    \leq
    \int_{0}^{T}
    \int_{-\infty}^{+\infty}
    \sum_{\indexVelocity}
    |\distributionFunctionLetter_{\indexVelocity}^{\vectorial{\varepsilon}}(\timeVariable, \spaceVariable)-\distributionFunctionLetter_{\indexVelocity}^{\atEquilibrium}(\kappa)|
    \discreteVelocityLetter_{\indexVelocity}
    \partial_{\spaceVariable}\psi(\timeVariable, \spaceVariable)
    \differential\spaceVariable\differential\timeVariable.
  \end{equation}
\end{theorem}
\begin{proof}
  The extraction of a converging subsequence as $\spaceStep\to 0$ (and $\timeStep=\spaceStep/\latticeVelocity\to 0$) can be done by the properties in \Cref{prop:propertiesEnsuringExtractionScheme}, and provides $\vectorial{f}^{\vectorial{\varepsilon}}$.
  Convergence also holds almost everywhere.
  This fact allows passing to the limit in the inequalities of \Cref{prop:propertiesEnsuringExtractionScheme} and \Cref{prop:deltaFromEquil}.
  Moreover, the limit is the solution of the Cauchy problem by consistency of the numerical scheme with it, and can be easily checked.
  Passing to the limit in the entropy inequality by \Cref{prop:discreteEntropy} (starred quantities merge into unstarred ones in the limit) gives the last part of the claim.
\end{proof}
\begin{remark}
    By \cite{MR1409928}, see \Cref{sec:gettingRidOfBV}, we know that a local unique solution to \eqref{system1}--\eqref{eq:equilibriaForm}--\eqref{f0} exists for initial data in $L^{\infty}$ only. 
    \Cref{thm:convergenceNumericsPlusEstimates} proves global existence (plus important estimates to pass to the limit in the relaxation parameters) in the bounded total variation framework.
\end{remark}

\section{Convergence of the relaxation system as $(\ea, \es)\to (0, 0)$}\label{sec:convergenceEps}

\begin{theorem}\label{thm:convergenceEps}
  Assume that $u^{\initial}\in L^{\infty}(\reals)\cap L^{1}(\reals)\cap \textnormal{BV}(\reals)$.
  Let $(\ea, \es) \to (0, 0)$, and---while this convergence takes place---let the assumptions of \Cref{mono1} hold true.
  Then, for any $T>0$, up to the extraction of subsequences, we have the convergence
  \begin{align}\label{eq:convergenceToEquil}
    u^{\vectorial{\varepsilon}}
    =
    \distributionFunctionLetter_{\zeroVelocitySymbol}^{\vectorial{\varepsilon}}
    +
    \distributionFunctionLetter_{\positiveVelocitySymbol}^{\vectorial{\varepsilon}}
    +
    \distributionFunctionLetter_{\negativeVelocitySymbol}^{\vectorial{\varepsilon}}
    &\to u, 
    \\ 
    v^{\vectorial{\varepsilon}}
    =
    \latticeVelocity(\distributionFunctionLetter_{\positiveVelocitySymbol}^{\vectorial{\varepsilon}}
    -
    \distributionFunctionLetter_{\negativeVelocitySymbol}^{\vectorial{\varepsilon}})
    &\to \flux(u), 
    \\ 
    w^{\vectorial{\varepsilon}}
    =
    \distributionFunctionLetter_{\positiveVelocitySymbol}^{\vectorial{\varepsilon}}
    +
    \distributionFunctionLetter_{\negativeVelocitySymbol}^{\vectorial{\varepsilon}}
    &\to 2\equilibriumCoefficientLinear  u,
    \qquad \qquad
    \text{in}
    \quad 
    C^0([0, T]; L^1_{\textnormal{loc}}(\reals)),
  \end{align}
  where $u$ is the weak entropy solution of the conservation law \eqref{eq:conservationLaw}.
\end{theorem}
\begin{proof}
  The extraction of a converging subsequence to some limit $\overline{\vectorial{\distributionFunctionLetter}}$ as $(\ea, \es) \to (0, 0)$ can be carried by the results in \Cref{thm:convergenceNumericsPlusEstimates}.
  By virtue of \eqref{eq:almostEquilibrium}, the limit coincides almost everywhere with the equilibrium, hence \eqref{eq:convergenceToEquil}.
  Moreover, this ensures, by taking the limit into the weak formulation of the first equation, that the limit $u$ is a solution to the scalar conservation law.
  To show that this is the sole entropy solution, let us take the limit $(\ea, \es) \to (0, 0)$ in \eqref{eq:krushkovFixedEps}:
  \begin{equation}
    -\int_{0}^{T}
    \int_{-\infty}^{+\infty}
    \sum_{\indexVelocity}
    |\distributionFunctionLetter_{\indexVelocity}^{\atEquilibrium}(u(\timeVariable, \spaceVariable))-\distributionFunctionLetter_{\indexVelocity}^{\atEquilibrium}(\kappa)|
    \partial_{\timeVariable}\psi(\timeVariable, \spaceVariable)
    \differential\spaceVariable\differential\timeVariable
    \leq
    \int_{0}^{T}
    \int_{-\infty}^{+\infty}
    \sum_{\indexVelocity}
    |\distributionFunctionLetter_{\indexVelocity}^{\atEquilibrium}(u(\timeVariable, \spaceVariable))-\distributionFunctionLetter_{\indexVelocity}^{\atEquilibrium}(\kappa)|
    \discreteVelocityLetter_{\indexVelocity}
    \partial_{\spaceVariable}\psi(\timeVariable, \spaceVariable)
    \differential\spaceVariable\differential\timeVariable.
  \end{equation}
  As in the proof of \cite[Theorem 4.11]{aregba2026monotonicity}, we obtain the entropy inequality with Krushkov entropies.
\end{proof}

This concludes the proof of the results which are recapitulated in \Cref{thm:hugeTheorem}.

\section{Extension to initial data of $L^{\infty}$-regularity only}\label{sec:gettingRidOfBV}

We now illustrate how to consider $u^{\initial}\in L^{\infty}(\reals)$ to obtain \Cref{thm:hugeTheoremLinf}.
We here assume without further mentioning it that the assumptions of \Cref{mono1} are fulfilled.
We proceed by density, noting that $u^{\initial}\in L^{1}_{\textnormal{loc}}(\reals)$.
We have a sequence of smooth (thus in particular of bounded total variation) and compactly supported functions $u^{\initial}_{\delta}$, with $\delta > 0$, such that 
\begin{equation*}
  u^{\initial}_{\delta}\xrightarrow[]{\delta \searrow 0} u^{\initial} \qquad \text{in}\quad L^1_{\textnormal{loc}}(\reals),
  \qquad \text{and}
  \qquad
  \lVert u^{\initial}_{\delta}\rVert_{\infty}\leq \lVert u^{\initial}\rVert_{\infty}.
\end{equation*}
We indicate 
\begin{itemize}
  \item by $u$, the solution of \eqref{eq:conservationLaw} with initial datum $u^{\initial}\in L^{\infty}(\reals)$.
  \item By ${u}_{\delta}$, the solution of \eqref{eq:conservationLaw} with regularized initial datum $u^{\initial}_{\delta}$.
  \item By $u^{\vectorial{\varepsilon}}$, the conserved moment of the relaxation system \eqref{system1}--\eqref{eq:equilibriaForm}--\eqref{f0} with initial datum $u^{\initial}\in L^{\infty}(\reals)$.
  \item By $u^{\vectorial{\varepsilon}}_{\delta}$, the conserved moment of the relaxation system \eqref{system1}--\eqref{eq:equilibriaForm}--\eqref{f0} with regularized initial datum $u^{\initial}_{\delta}$.
\end{itemize}
We emphasize the fact that, at this stage, there is \strong{no guarantee} that global existence of $u^{\vectorial{\varepsilon}}$ holds true.
We aim at showing that it is indeed the case.
For the moment, just assume this fact: we obtain, for any compact set
$K\subset \reals$ and for any time $T>0$ of existence of
$u^{\vectorial{\varepsilon}}$
\begin{multline*}
  \sup_{t\in[0, T]}\lVert u^{\vectorial{\varepsilon}}(t, \cdot) - u (t, \cdot)\rVert_{L^1(K)}
  \\
  \leq 
  \sup_{t\in[0, T]}\lVert u^{\vectorial{\varepsilon}}(t, \cdot) - u^{\vectorial{\varepsilon}}_{\delta}(t, \cdot)\rVert_{L^1(K)}+
  \sup_{t\in[0, T]}\lVert u^{\vectorial{\varepsilon}}_{\delta}(t, \cdot) - {u}_{\delta}(t, \cdot)\rVert_{L^1(K)}+
  \sup_{t\in[0, T]}\lVert {u}_{\delta}(t, \cdot) - u(t, \cdot)\rVert_{L^1(K)}.
\end{multline*}
Let us discuss term by term.
\begin{itemize}
  \item The first term is such that
  \begin{equation*}
    \sup_{t\in[0, T]}\lVert u^{\vectorial{\varepsilon}}(t, \cdot) - u^{\vectorial{\varepsilon}}_{\delta}(t, \cdot)\rVert_{L^1(K)}
    \leq C_{K, T}
    \lVert u^{\initial} - u^{\initial}_{\delta}\rVert_{L^1(K)}
    \xrightarrow[]{\delta \searrow 0}
    0,
  \end{equation*}
  where the first inequality comes from Proposition 2.1 in \cite{MR1409928}.
  \item The second term is such that 
  \begin{equation*}
    \sup_{t\in[0, T]}\lVert u^{\vectorial{\varepsilon}}_{\delta}(t, \cdot) - {u}_{\delta}(t, \cdot)\rVert_{L^1(K)}
    \xrightarrow[]{\varepsilon \searrow 0}
    0,
  \end{equation*}
  thanks to \Cref{thm:convergenceEps}, which holds in the bounded total variation setting.
  \item For the third term, we have 
  \begin{equation*}
    \sup_{t\in[0, T]}\lVert {u}_{\delta}(t, \cdot) - u(t, \cdot)\rVert_{L^1(K)}
    \leq
    \lVert {u}^{\initial}_{\delta} - u^{\initial}\rVert_{L^1(K)}
    \xrightarrow[]{\delta \searrow 0}
    0,
  \end{equation*}
  where the first inequality comes from the $L^1$ contractivity of the semi-group associated to \eqref{eq:conservationLaw}, see Theorem 5.1 in \cite{godlewski1991hyperbolic}.
\end{itemize}

We are left to prove that $\vectorial{f}^{\vectorial{\varepsilon}}$ (and thus ${u}^{\vectorial{\varepsilon}}$), solution of the relaxation system \eqref{system1}--\eqref{eq:equilibriaForm}--\eqref{f0} with initial datum $u^{\initial}\in L^{\infty}(\reals)$, exists globally, that is on $[0, T]$ for any $T>0$.
To this end, we need to recall some results on \strong{quasi-linear} systems, see \cite{MR1409928} and \cite{MR1643175}.
Local existence is ensured by the following, which is \cite[Theorems 3.1 and 3.2]{MR1409928} and \cite[Proposition 3.4]{MR1643175}
\begin{proposition}[Local existence and continuation]\label{prop:localExistenceQuasiLinear}
  For any initial datum $\vectorial{f}^{\initial}\in L^{\infty}(\reals)^3$, there exists $T=T(\lVert\vectorial{f}^{\initial}\rVert_{\infty})>0$ such that there exists a unique (weak) solution $\vectorial{f}^{\vectorial{\varepsilon}}$ of \eqref{system1}--\eqref{eq:equilibriaForm} on $(0, T)\times \reals$ with $\vectorial{f}^{\vectorial{\varepsilon}} \in C^0([0, T); L^1_{\textnormal{loc}}(\reals)^3)$.
  Moreover, 
  \begin{itemize}
    \item either $\vectorial{f}^{\vectorial{\varepsilon}} \in L^{\infty}((0, T)\times \reals)^3$ for any $T>0$ (\idEst{}, the solution is global);
    \item or there exists $T^{*}>0$ such that, for any $T<T^*$, $\vectorial{f}^{\vectorial{\varepsilon}}$ is defined on $(0, T)\times \reals$, and 
    \begin{equation*}
      \lim_{T\nearrow T^*} \lVert \vectorial{f}^{\vectorial{\varepsilon}} \rVert_{L^{\infty}((0, T)\times \reals)^3} = +\infty 
      \qquad \text{(blowup)}.
    \end{equation*}
  \end{itemize}
\end{proposition}
The previous result tells us that in order to prove that the solution $\vectorial{f}^{\vectorial{\varepsilon}}$ is global, we need to show that it \strong{does not blow-up} in finite time.
This can be done by showing that it remains (almost everywhere) in a bounded set.
This can be shown using a \strong{comparison principle}, see \cite[Proposition 3.5]{MR1643175}:
\begin{proposition}[Comparison principle]\label{prop:comparison}
  Let $\vectorial{f}^{\vectorial{\varepsilon}}$ and $\tilde{\vectorial{f}}^{\vectorial{\varepsilon}}$ be two solutions of  \eqref{system1}--\eqref{eq:equilibriaForm}on $(0, T)\times \reals$ with initial data $\vectorial{f}^{\initial}$ and $\tilde{\vectorial{f}}^{\initial}$.
  Let $Q\subset \reals^3$ be an interval with non-empty interior, such that 
  \begin{itemize}
    \item the right-hand side $\vectorial{G}$ of \eqref{system1} is quasi-monotone on $Q$.
    \item $\vectorial{f}^{\vectorial{\varepsilon}}, \tilde{\vectorial{f}}^{\vectorial{\varepsilon}}\in Q$ almost everywhere in  $(0, T)\times \reals$.
  \end{itemize}
  Then, if $\vectorial{f}^{\initial}\leq \tilde{\vectorial{f}}^{\initial}$ almost everywhere in $\reals$, then $\vectorial{f}^{\vectorial{\varepsilon}}\leq \tilde{\vectorial{f}}^{\vectorial{\varepsilon}}$ almost everywhere in  $(0, T)\times \reals$.
\end{proposition}
Note that the constant function $\tilde{\vectorial{f}}^{\vectorial{\varepsilon}}(\timeVariable, \spaceVariable) =\tilde{\vectorial{f}}^{\initial}(\spaceVariable)= \vectorial{\minimumDistribution}$ (respectively $\tilde{\vectorial{f}}^{\vectorial{\varepsilon}}(\timeVariable, \spaceVariable) = \tilde{\vectorial{f}}^{\initial}(\spaceVariable)=\vectorial{\maximumDistribution}$) is a (global) solution to  \eqref{system1}--\eqref{eq:equilibriaForm}, thus is the one we use when invoking \Cref{prop:comparison}.
To employ \Cref{prop:comparison}, we however need to know \strong{a priori} that the ``other'' solution $\vectorial{f}^{\vectorial{\varepsilon}}$ belongs to some interval $Q$, where quasi-monotonicity must hold.
Thus, in order to avoid a ``circular reasoning'', the only choice is to take $Q = \reals^3$. 
We follow a procedure inspired by \cite{MR1643175}: consider partly ``\strong{frozen}'' equilibria when their argument lays outside $[-\maximumInitialDatum, \maximumInitialDatum]$, where $\maximumInitialDatum = \lVert u^{\initial}\rVert_{\infty}$.
This reads 
\begin{equation}
\distributionFunctionLetter_{\zeroVelocitySymbol}^{\atEquilibrium, \#}(u) \definitionEquality
\distributionFunctionLetter_{\zeroVelocitySymbol}^{\atEquilibrium}(u)  = (1-2\equilibriumCoefficientLinear ) u
\qquad \text{and} \qquad
\distributionFunctionLetter_{\pm}^{\atEquilibrium, \#}(u) \definitionEquality
\begin{cases}
  \equilibriumCoefficientLinear  u \pm \frac{\flux(-\maximumInitialDatum)}{2\latticeVelocity}, \qquad &\text{if}\quad u<-\maximumInitialDatum,\\
  \equilibriumCoefficientLinear  u \pm \frac{\flux(\conservedVariable)}{2\latticeVelocity}, \qquad &\text{if}\quad |u|\leq \maximumInitialDatum, \\
  \equilibriumCoefficientLinear  u \pm \frac{\flux(\maximumInitialDatum)}{2\latticeVelocity}, \qquad &\text{if}\quad u>\maximumInitialDatum.
\end{cases}
\end{equation}
Note that different choices can be done. Here, we decided to freeze the non-linear part of the equilibria.
We denote by $\vectorial{G}^{\#}$ the right-hand side of \eqref{system1} where the functions $\distributionFunctionLetter_{\indexVelocity}^{\atEquilibrium}$ have been replaced by $\distributionFunctionLetter_{\indexVelocity}^{\atEquilibrium, \#}$. 
This function coincides with $\vectorial{G}$ for arguments $\vectorial{f}^{\vectorial{\varepsilon}}$ such that $|\sum_{\indexVelocity}{f}_{\indexVelocity}^{\vectorial{\varepsilon}}|\leq \maximumInitialDatum$, and for such arguments the same usual sufficient quasi-monotonicity conditions can be employed.
Now let $\vectorial{f}^{\vectorial{\varepsilon}}$ such that $|\sum_{\indexVelocity}{f}_{\indexVelocity}^{\vectorial{\varepsilon}}|> \maximumInitialDatum$: we do not study $G_{\zeroVelocitySymbol}^{\#}$ since it is unchanged. 
On the other hand, $\partial_{f_{\zeroVelocitySymbol}}G_{\pm}^{\#} = \frac{1}{\es} \equilibriumCoefficientLinear$, which commands $\equilibriumCoefficientLinear\geq 0$. Finally 
\begin{equation*}
  \partial_{f_{\mp}}G_{\pm}^{\#}
  =
  \frac{1}{2}
  \Bigl ( 
    \frac{1}{\es}+\frac{1}{\ea}
  \Bigr )\equilibriumCoefficientLinear
  +
  \frac{1}{2}
  \Bigl ( 
    \frac{1}{\es}-\frac{1}{\ea}
  \Bigr )(\equilibriumCoefficientLinear-1)
  \qquad \text{hence we want}
  \qquad 
  \equilibriumCoefficientLinear\geq \frac{1}{2}
  \Bigl( 
    1-\frac{\es}{\ea}
    \Bigr ).
\end{equation*}
We have thus proved the following result.
\begin{proposition}\label{prop:quasiMonotonicityFrozenEquil}
  Assume that 
  \begin{align*}
    \text{if}\quad
    \es<\ea, 
    \quad 
    &\text{then}
    \quad 
    \equilibriumCoefficientLinear\in 
    \Bigl [
    1-\frac{\es}{\ea}
    , \frac{1}{2}\Bigr ]; \quad \text{or} \\
    \text{if}\quad
    \es\geq \ea, 
    \quad 
    &\text{then}
    \quad 
    \equilibriumCoefficientLinear\in 
    \Bigl (0
    , \frac{1}{2}\Bigr ],
  \end{align*}
  and
  \begin{align}
    &\text{if}\quad \equilibriumCoefficientLinear   = \tfrac{1}{2}, 
    \quad \text{then}
    \quad 
    \max_{u\in[-\maximumInitialDatum, \maximumInitialDatum]}
    \frac{|\flux^\prime(u)|}{\lambda}\leq 1; \qquad \text{or} \\
    &\text{if}\quad \equilibriumCoefficientLinear   \neq \tfrac{1}{2}, 
    \quad \text{then}
    \quad 
    \max_{u\in[-\maximumInitialDatum, \maximumInitialDatum]}
    \frac{|\flux^\prime(u)|}{\lambda}\leq \min\left(  \frac{2  \ea \equilibriumCoefficientLinear  }{\es},1-\frac{\ea}{\es}(1-2\equilibriumCoefficientLinear  )\right).
\end{align}
Under these conditions, the function $\vectorial{G}^{\#}$ is quasi-monotone on $Q = \reals^3$.
\end{proposition}

\begin{proposition}
  Let $u^{\initial}\in L^{\infty}(\reals)$ and let the assumptions of \Cref{prop:quasiMonotonicityFrozenEquil} be fulfilled.
  Then, for any $T>0$, there exists a unique $\vectorial{f}^{\vectorial{\varepsilon}}$ (and thus ${u}^{\vectorial{\varepsilon}}$) solution of the relaxation system \eqref{system1}--\eqref{eq:equilibriaForm}--\eqref{f0} on $(0, T]\times \reals$ with $\vectorial{f}^{\vectorial{\varepsilon}}\in C^0([0, T]; L^1_{\textnormal{loc}}(\reals)^3)$ and $\vectorial{f}^{\vectorial{\varepsilon}}
    \in [\vectorial{\minimumDistribution}, \vectorial{\maximumDistribution}]$  (and thus $u^{\vectorial{\varepsilon}}\in [-\maximumInitialDatum, \maximumInitialDatum]$) almost everywhere.
\end{proposition}
\begin{proof}
  Note that \Cref{prop:localExistenceQuasiLinear} applies also when the right-hand side $\vectorial{G}$ of \eqref{system1} is replaced by $\vectorial{G}^{\#}$.
  So, let us denote $\vectorial{f}^{\vectorial{\varepsilon}, \#}$ the unique solution of \eqref{system1}--\eqref{eq:equilibriaForm} with initial datum 
  \begin{equation*}
    \vectorial{f}^{\vectorial{\varepsilon}, \#}(\timeVariable = 0, \spaceVariable)
    =
    \vectorial{f}^{\atEquilibrium, \#}(u^{\initial}(\spaceVariable))
    =
    \vectorial{f}^{\atEquilibrium}(u^{\initial}(\spaceVariable)),
  \end{equation*}
  along with its existence time $T^{\#}>0$.
  By virtue of \Cref{prop:quasiMonotonicityFrozenEquil}, we can use \Cref{prop:comparison} comparing $\vectorial{f}^{\vectorial{\varepsilon}, \#}$ with constant solutions, and obtain
  \begin{equation*}
    \vectorial{f}^{\vectorial{\varepsilon}, \#}(\timeVariable, \spaceVariable)
    \in [\vectorial{\minimumDistribution}, \vectorial{\maximumDistribution}]
    \qquad 
    \text{for almost every}
    \quad
    (\timeVariable, \spaceVariable)\in [0, T^{\#})\times \reals.
  \end{equation*}
  This entails, by the way, that $\vectorial{f}^{\vectorial{\varepsilon}, \#}\equiv \vectorial{f}^{\vectorial{\varepsilon}}$ on $[0, T^{\#})\times \reals$, with the latter being the (local) solution of \eqref{system1}--\eqref{eq:equilibriaForm}--\eqref{f0}.
  As the solution remains uniformly bounded, we can invoke a continuation argument and argue by the second part of \Cref{prop:localExistenceQuasiLinear}: $\vectorial{f}^{\vectorial{\varepsilon}}$ and thus $u^{\vectorial{\varepsilon}}$ is global.
\end{proof}

\section{Conclusions}

In this work, we have introduced a new relaxation approximation for a scalar one-dimensional conservation law featuring two relaxation parameters and three discrete velocities.
This relaxation system is directly inspired by well-established works on lattice Boltzmann schemes, which can be somehow regarded as its discrete counterpart.
In our way of proceeding, not only the numerical scheme is an inspiration to build the relaxation approximation to the original PDE, but is also used to deduce properties of the solutions of the relaxation approximation that allow to prove its convergence to the solution of the PDE.

Extensions of the work to the multi-dimensional setting and/or to larger discrete velocity stencils are straightforward, see \cite{aregba2026monotonicity}, provided that one consider couples of pair-wise opposite discrete velocities and that the relaxation term has the same structure as in \eqref{system1}.

\bibliographystyle{alpha}
\bibliography{biblio}

\appendix

\section{Equivalent second-order equation on ${u}$}\label{app:equivalentSecondU}

We rewrite \eqref{eq:systemODEuandV} as
\begin{equation*}
    \begin{pmatrix}
        -c^2 \textnormal{d}_{\xi} + \frac{c}{\ea} & \lambda^2\textnormal{d}_{\xi}\\
        c \textnormal{d}_{\xi} - \frac{2\equilibriumCoefficientLinear }{\es} & -c\textnormal{d}_{\xi} + \frac{1}{\es}
    \end{pmatrix}
    \begin{pmatrix}
        {u}(\xi)\\
        {w}(\xi)
    \end{pmatrix}
    =
    \begin{pmatrix}
        \frac{1}{\ea}(\flux({u}(\xi))-C_v)\\
        0
    \end{pmatrix}.
\end{equation*}
With this
\begin{equation*}
    \textnormal{det}
    \begin{pmatrix}
        -c^2 \textnormal{d}_{\xi} + \frac{c}{\ea} & \lambda^2\textnormal{d}_{\xi}\\
        c \textnormal{d}_{\xi} - \frac{2\equilibriumCoefficientLinear }{\es} & -c\textnormal{d}_{\xi} + \frac{1}{\es}
    \end{pmatrix}
    =
    c(c^2 - \lambda^2) \textnormal{d}_{\xi\xi} 
    + \Bigl( -\frac{c^2}{\es} - \frac{c^2}{\ea} + \frac{2\lambda^2\equilibriumCoefficientLinear }{\es}\Bigr )\textnormal{d}_{\xi}
    +\frac{c}{\ea\es},
\end{equation*}
and 
\begin{equation*}
    \transpose{\bm{e}_1}
    \textnormal{\textbf{adj}}
    \begin{pmatrix}
        -c^2 \textnormal{d}_{\xi} + \frac{c}{\ea} & \lambda^2\textnormal{d}_{\xi}\\
        c \textnormal{d}_{\xi} - \frac{2\equilibriumCoefficientLinear }{\es} & -c\textnormal{d}_{\xi} + \frac{1}{\es}
    \end{pmatrix}
    = 
    \Bigl (-c\textnormal{d}_{\xi} + \frac{1}{\es}, -\lambda^2\textnormal{d}_{\xi} \Bigr ).
\end{equation*}
This entails---by virtue of the Laplace expansion of the determinant---that the differential equation on ${u}$ reads as in \eqref{eq:secondOrderU}. 

\section{Proof of \Cref{lemma:classificationEquilibriaBurgers}}\label{app:lemma:classificationEquilibriaBurgers}

Under the assumptions in the claim, we have that $\gamma<0$.
  Moreover, $\textnormal{det}\mathscr{J}(u, 0) = \frac{\gamma}{c}(c-u) = \frac{\gamma}{c}(c-\flux'(u))$, hence by the assumptions $\textnormal{det}\mathscr{J}(u_+, 0)<0$ and $\textnormal{det}\mathscr{J}(u_-, 0)>0$.
  This entails that the eigenvalues for $u_+$ have real parts of opposite sign and are indeed real.
  Those for $u_-$ have real parts of same sign.

  Let us now further investigate $u_-$: we have
  \begin{equation*}
    \textnormal{tr} \mathscr{J}(u_-, 0)
    =
  \frac{1}{\lambda^2-c^2}
\Bigl ( \frac{2\lambda^2\equilibriumCoefficientLinear }{\es c}   -\frac{u_- + u_+}{2\es} + \frac{u_- - u_+}{2\ea}\Bigr ).
  \end{equation*}
  If $\textnormal{tr} \mathscr{J}(u_-, 0)<0$, the eigenvalues have both negative real part (the equilibrium is a sink), and the heteroclinic profile cannot exist.
  Indeed, even reversing $\xi\mapsto-\xi$, this equilibrium becomes a source, hence cannot be connected to the other in the reversed direction.
  This condition boils down to requesting 
  \begin{equation*}
    \equilibriumCoefficientLinear  <
    \frac{c}{2\lambda^2 }
    \Bigl ( 
    \frac{u_- + u_+}{2} -\frac{u_- - u_+}{2}\frac{\es}{\ea}
    \Bigr ),
  \end{equation*}
  which is a way of saying the dissipation is extremely low in the system.
  This can be attained provided that the term into parentheses is positive.
  This reads 
  \begin{equation*}
    \frac{\es}{\ea} 
    <
    \frac{u_- + u_+}{u_- - u_+}.
  \end{equation*}
  In the case $\textnormal{tr} \mathscr{J}(u_-, 0) = 0$, the equilibrium is not hyperbolic and the Hartman-Grobman theorem does not provide information. 
  For $\textnormal{tr} \mathscr{J}(u_-, 0)>0$, the equilibrium is a source, which is necessary to obtain the existence of a heteroclinic orbit.
  One can study if the equilibrium is a focus or a node, namely in the first case, the eigenvalues possess a non-zero imaginary part.
  This boils down to checking $\textnormal{tr} \mathscr{J}(u_-, 0)^2<4\, \textnormal{det} \mathscr{J}(u_-, 0)$.
  Since we are in the case of positive trace, we study $\textnormal{tr} \mathscr{J}(u_-, 0)<2 \sqrt{\textnormal{det} \mathscr{J}(u_-, 0)}$, which reads
  \begin{equation*}
\equilibriumCoefficientLinear 
 < 
 \frac{c}{2\lambda^2}
 \Bigl ( \frac{u_- + u_+}{2} - \frac{u_- - u_+}{2}\frac{\es}{\ea} + 2 \sqrt{(\lambda^2-c^2)
  \frac{\es}{\ea}\frac{u_- - u_+}{u_-+u_+}
}\Bigr ).
  \end{equation*}

\end{document}